\documentclass[10pt,a4paper]{amsart}

\usepackage{color}

\usepackage[utf8]{inputenc}
\usepackage{stmaryrd}

\newcommand{\xint}{\mathrm{int}}
\newcommand{\intDop}{\displaystyle{\int_{\Dop}}\!\!\!\!\!\!}

\newcommand{\itcatl}{$(\infty,2)$-categorical}

\usepackage[sec,skipfonts,skipmbb]{enrabbrev}

\theoremstyle{definition}

\newtheorem{assumption}[thm]{Assumption}

\usepackage{eucal}

\title[On lax transformations, adjunctions, and monads in $(\infty,2)$-categories]{On lax transformations, adjunctions,\\ and monads in $(\infty,2)$-categories}

\newcommand{\Dop}{\simp^{\op}}

\newcommand{\Adj}{\txt{Adj}}

\newcommand{\End}{\txt{End}}

\newcommand{\lax}{\txt{lax}}

\newcommand{\colax}{\txt{colax}}
\newcommand{\ladj}{\txt{ladj}}
\newcommand{\radj}{\txt{radj}}
\newcommand{\mndradj}{\txt{mndradj}}

\newcommand{\LMod}{\txt{LMod}}

\newcommand{\fmnd}{\mathfrak{mnd}}

\newcommand{\Nat}{\txt{Nat}}

\renewcommand{\opctriangle}[6]{
\[%
\begin{tikzcd}%
  #1 \arrow[swap]{dr}{#5} \arrow{rr}{#4} \pgfmatrixnextcell \pgfmatrixnextcell #2 \arrow{dl}{#6} \\%
  {} \pgfmatrixnextcell #3 %
\end{tikzcd}%
\] %
}

\newcommand{\ie}{i.e.\@}

\newcommand{\CatIT}{\Cat_{(\infty,2)}}
\newcommand{\pcolax}{\txt{(co)lax}}
\newcommand{\ADJ}{\txt{ADJ}}
\newcommand{\MND}{\txt{MND}}
\newcommand{\END}{\txt{END}}
\newcommand{\fadj}{\mathfrak{adj}}
\newcommand{\fend}{\mathfrak{end}}
\newcommand{\igpd}{$\infty$-groupoid}
\newcommand{\igpds}{$\infty$-groupoids}
\newcommand{\twop}{\txt{2-op}}
\newcommand{\onop}{\txt{1-op}}
\newcommand{\obj}{\txt{ob}}
\newcommand{\bTt}{\bbTheta_{2}}

\usepackage{xstring}
\renewcommand{\eprint}[1]{\IfBeginWith{#1}{arXiv}{\href{https://arxiv.org/abs/#1}{#1}}{\href{#1}{#1}}}

\begin{document}

\begin{abstract}
  We use the basic expected properties of the Gray tensor product of
  $(\infty,2)$-categories to study (co)lax natural
  transformations. Using results of Riehl--Verity and Zaganidis we
  identify lax transformations between adjunctions and monads with
  commutative squares of (monadic) right adjoints. We also identify
  the colax transformations whose components are equivalences
  (generalizing the ``icons'' of Lack) with the 2-morphisms that arise
  from viewing $(\infty,2)$-categories as simplicial
  $\infty$-categories. Using this characterization we identify the
  $\infty$-category of monads on a fixed object and colax morphisms
  between them with the $\infty$-category of associative algebras in
  endomorphisms.
\end{abstract}
\maketitle

\tableofcontents

\section{Introduction}
Consider the following descriptions of \emph{monads} on an
\icat{}:
\begin{enumerate}[(A)]
\item\label{it:alg} A monad on $\mathcal{C}$ is an associative algebra in the monoidal \icat{}
  $\Fun(\mathcal{C}, \mathcal{C})$ of endofunctors under composition.
\item\label{it:mra} A monad is a functor of \icats{} that is a monadic right adjoint.
\item\label{it:mtf} A monad is a functor of \itcats{} $\fmnd \to \CATI$, where
  $\fmnd$ is the universal 2-category containing a monad and $\CATI$
  is the \itcat{} of \icats{}.
\end{enumerate}
These three definitions are known to be equivalent by results of
Lurie~\cite{HA}*{\S 4.7.3} and
Riehl--Verity~\cite{RiehlVerityAdj}.\footnote{Both prove versions of
  the monadicity theorem, which relate the first two and last two
  definitions, respectively.} However, these comparisons only relate
\igpds{} of monads. Our main goal in this paper is to enhance the
comparisons to take into account \emph{morphisms} of monads. For
\ref{it:alg} the obvious notion of morphism between monads on
$\mathcal{C}$ is a homomorphism of algebras in
$\Fun(\mathcal{C},\mathcal{C})$, while for \ref{it:mra} it is a
commutative triangle
\[
  \begin{tikzcd}
    \mathcal{A} \arrow{rr}{f} \arrow[swap]{dr}{r}  & & \mathcal{B}
    \arrow{dl}{r'} \\
     & \mathcal{C}
  \end{tikzcd}
  \]
where $r$ and $r'$ are monadic right adjoints.\footnote{On a fixed
  \icat{} $\mathcal{C}$, monads in sense \ref{it:alg} and \ref{it:mra}
  have already been compared by Heine~\cite{Heine}.} More generally, we can
allow the \icat{} $\mathcal{C}$ to vary and consider commutative
squares whose vertical morphisms are monadic right adjoints.

For ordinary 2-categories, Street~\cite{StreetFormalMonad} showed that
such squares of monadic right adjoints correspond to what he called
\emph{monad functors}, which are the same thing as \emph{lax natural
  transformations} between functors from $\fmnd$. To compare
\ref{it:mra} and \ref{it:mtf} we therefore start by studying lax
transformations in the setting of \itcats{}. These can be defined
using the \emph{(lax) Gray tensor product}. This has not yet been
fully developed for \itcats{}\footnote{Though several constructions
  have recently appeared, and this is a topic of active research; see \cref{rmk:grayconstr}.}, and we do not do
so here. Instead, we assume it has certain basic expected properties
(see \cref{ass:gray}) and proceed from there to define \itcats{}
$\FUN(\mathcal{Y}, \mathcal{X})_{\pcolax}$ of functors and (co)lax
transformations between \itcats{} $\mathcal{Y}$ and $\mathcal{X}$ in
\S\ref{sec:gray} after briefly reviewing some descriptions of
\itcats{} in \S\ref{subsec:itcat}.  Specializing $\mathcal{Y}$ to the
universal monad 2-category $\fmnd$ and the universal adjunction
2-category $\fadj$ we obtain \itcats{} $\MND(\mathcal{X})_{\pcolax}$
and $\ADJ(\mathcal{X})_{\pcolax}$ of, respectively, monads and
adjunctions, with (co)lax transformations as morphisms. Our comparison
of \ref{it:mra} and \ref{it:mtf} is then the combination of the
following two results:
\begin{thm}\label{thm:adjlaxradj}
  For any \itcat{} $\mathcal{X}$, restricting an adjunction to its right adjoint defines an
  equivalence of \itcats{}
  \[ \ADJ(\mathcal{X})_{\lax} \isoto \FUN(C_{1}, \mathcal{X})_{\radj},\]
  where the latter is the \itcat{} of morphisms in $\mathcal{X}$ that
  are right adjoints, with commutative squares as morphisms.
\end{thm}
\begin{thm}\label{thm:mndlaxadj}
  The functor $\ADJ(\CATI)_{\lax} \to \MND(\CATI)_{\lax}$ taking an
  adjunction to its induced monad, has a fully faithful right adjoint
  with image the monadic adjunctions.
\end{thm}
We prove \cref{thm:adjlaxradj} in \S\ref{sec:adj}, using one of the
main results of \cite{RiehlVerityAdj}, which gives this equivalence on
the level of underlying \igpds{}. \cref{thm:mndlaxadj} is then proved
in \S\ref{sec:mnd} as a corollary of work of
Zaganidis~\cite{Zaganidis}, whose thesis studied lax morphisms of
adjunctions and monads in the framework of
\cite{RiehlVerityAdj}. Combining these two theorems we get an
equivalence of \itcats{}
\[ \MND(\CATI)_{\lax} \simeq \FUN(C_{1}, \CATI)_{\mndradj},\]
where the right-hand side is the \itcat{} of morphisms in
$\CATI$ that are monadic right adjoints. (More generally, we can
replace $\CATI$ by any \itcat{} that can be modelled by an
\emph{$\infty$-cosmos} in the sense of Riehl and Verity.)

We then turn to the relation between descriptions \ref{it:alg} and
\ref{it:mtf}. To see that these give the same objects it is enough to
observe that the one-object 2-category $\fmnd$ is the monoidal
envelope of the non-symmetric associative operad, but to relate the
morphisms we need to understand the connection between (co)lax
transformations and 2-morphisms of monoidal
\icats{}. More generally, if we view \itcats{} (in the guise of
complete 2-fold Segal spaces) as cocartesian
fibrations over $\Dop$, then for \itcats{} $\mathcal{X}$ and
$\mathcal{Y}$ we can define an \icat{} $\Nat(\mathcal{X}, \mathcal{Y})$
consisting of functors over $\Dop$ that preserve cocartesian morphisms
and natural transformations between them. In
\S\ref{sec:icon} we prove the following characterization of these \icats{}:
\begin{thm}
  There is a functor
  \[\Nat(\mathcal{X},\mathcal{Y}) \to
    \Fun(\mathcal{X},\mathcal{Y})_{\colax}\] that identifies the
  domain with the wide subcategory of the \icat{}
  $\Fun(\mathcal{X}, \mathcal{Y})_{\colax}$ underlying
  $\FUN(\mathcal{X}, \mathcal{Y})_{\colax}$ containing those colax
  transformations whose components are all equivalences.
\end{thm}
The colax transformations in this subcategory are an \itcatl{}
analogue of the ``icons'' of Lack~\cite{Lack:icons}. Combining this result with the
non-symmetric analogues of the results on (symmetric) monoidal
envelopes of \iopds{} from \cite{HA}*{\S 2.2.4}, we obtain the
following comparison of descriptions \ref{it:alg} and
\ref{it:mtf} in \S\ref{sec:mndalg}:
\begin{thm}
  For any object $X$ of an \itcat{} $\mathcal{X}$, there is an
  equivalence of \icats{}
  \[ \Alg(\End_{\mathcal{X}}(X)) \isoto
    \Mnd(\mathcal{X})_{\colax,X} \]
  between the \icat{} of associative algebras in the monoidal \icat{} of
  endomorphisms of $X$ under composition, and the fibre at $X$ of the
  underlying \icat{} $\Mnd(\mathcal{X})_{\colax}$ of $\MND(\mathcal{X})_{\colax}$.
\end{thm}
This equivalence is compatible with the forgetful functors to
endomorphisms of $X$. Replacing lax by colax morphisms, we also obtain
an equivalence between $\Alg(\End_{\mathcal{X}}(X))^{\op}$ and
$\Mnd(\mathcal{X})_{\lax,X}$ and so combined with our first comparison
we obtain for $\mathcal{C}$ an \icat{} equivalences
\[ \Alg(\Fun(\mathcal{C}, \mathcal{C}))^{\op} \simeq
  \Mnd(\CATI)_{\lax, \mathcal{C}} \simeq
  \Cat_{\infty/\mathcal{C}}^{\mndradj} \]
where the right-hand side is the full subcategory of
$\Cat_{\infty/\mathcal{C}}$ spanned by the monadic right adjoints.

\subsection*{Acknowledgments}
This paper began as a revision of the appendix of \cite{polynomial},
and I thank David Gepner and Joachim Kock for a fruitful and rewarding
collaboration. I especially thank Joachim for extensive discussions on
the subject of this paper, particularly Gray tensor products and lax
transformations. I also thank Alexander Campbell for supplying some
2-categorical references.

This paper was begun while the author was employed by the IBS Center
for Geometry and Physics in Pohang, in a position funded by grant
IBS-R003-D1 of the Institute for Basic Science of the Republic of
Korea; it was completed while the author was in residence at the
Matematical Sciences Research Institute in Berkeley, California,
during the Spring 2020 semester, and is thereby based upon work
supported by the National Science Foundation under grant DMS-1440140.

\section{$(\infty,2)$-Categories}\label{subsec:itcat}
In this section we fix some notation for various structures related to
\icats{}, and briefly review the different descriptions of \itcats{}
we make use of; we also make a few simple \itcatl{} observations that will be
useful later on.

\begin{notation}
  We write $\mathcal{S}$ for the \icat{} of spaces (or \igpds{}),
  $\CatI$ for the \icat{} of \icats{}, and $\CatIT$ for the \icat{} of
  \itcats{}.
\end{notation}

\begin{notation}
  If $\mathcal{C}$ is an \icat{}, we write $\mathcal{C}^{\simeq}$ for
  the underlying \igpd{} of $\mathcal{C}$, which is the value at
  $\mathcal{C}$ of the right adjoint to the inclusion $\mathcal{S}
  \hookrightarrow \CatI$. This inclusion also has a left adjoint,
  which takes the \icat{} $\mathcal{C}$ to the \igpd{} obtained by
  inverting all morphisms in $\mathcal{C}$, which we denote by $\| \mathcal{C}\|$.
\end{notation}

The \icat{} $\CatIT$ admits several useful descriptions; in particular, we can view
$(\infty,2)$-categories
\begin{itemize}
\item as complete 2-fold Segal spaces \cite{BarwickThesis},
\item as complete Segal $\bbTheta_{2}$-spaces \cite{RezkThetaN},
\item as certain simplicial objects in $\CatI$
  \cite{LurieGoodwillie},
\item or as \icats{} enriched in $\CatI$ \cite{enr}.
\end{itemize}
The first three of these descriptions
are related through the following commutative diagram, where all
functors except the lower right one are fully faithful:
\begin{equation}
  \label{eq:segdiag}
  \begin{tikzcd}
{} &
    \Seg_{\Dop}^{\mathcal{S}}(\CatI) \arrow[hookrightarrow]{r}
    \arrow[hookrightarrow]{d} & \Seg_{\Dop}(\CatI)
    \arrow[hookrightarrow]{r}  \arrow[hookrightarrow]{d} & \Fun(\Dop,
    \CatI) \arrow[hookrightarrow]{d} \\
    \CatIT \arrow[hookrightarrow]{ur}  \arrow[hookrightarrow]{dr}
    \arrow[hookrightarrow]{r}
    & \Seg^{\txt{2-fold}}_{\Dop \times \Dop}(\mathcal{S})
    \arrow[hookrightarrow]{r}  & \Seg_{\Dop \times \Dop}(\mathcal{S})
    \arrow[hookrightarrow]{r} & \Fun(\Dop \times \Dop, \mathcal{S})\\
    & \Seg_{\bbTheta_{2}^{\op}}(\mathcal{S})
    \arrow[hookrightarrow]{rr} \arrow{u}{\sim} & & 
    \Fun(\bTt^{\op}, \mathcal{S}) \arrow{u}{\tau^{*}}.
  \end{tikzcd}
\end{equation}
We now describe the \icats{} and functors that appear in this diagram:
\begin{defn}
  We write $\simp$ for the usual simplex category, consisting of the
  ordered sets $[n] := \{0,\ldots,n\}$ and order-preserving maps
  between them. A morphism $\phi \colon [m] \to [n]$ in $\simp$ is
  called \emph{inert} if it is the inclusion of a subinterval, \ie{}
  if $\phi(i) = \phi(0)+i$ for $i = 0,\ldots,m$, and \emph{active} if
  it preserves the end points, \ie{} $\phi(0) = 0$ and $\phi(m) = n$.
\end{defn}

\begin{defn}
  For an \icat{} $\mathcal{C}$ with finite limits,
  $\Seg_{\Dop}(\mathcal{C})$ denotes the full subcategory of
  $\Fun(\Dop, \mathcal{C})$ consisting of functors
  $X \colon \Dop \to \mathcal{C}$
  satisfying the \emph{Segal condition}, meaning that the natural map
  \[ X_{n} \to X_{1} \times_{X_{0}} \cdots \times_{X_{0}} X_{1},\]
  induced by the inert maps $[0],[1] \to [n]$ in $\simp$,
  is an equivalence for all $n$. We  also write $\Seg_{\Dop
    \times \Dop}(\mathcal{C})$ for the full subcategory
  $\Seg_{\Dop}(\Seg_{\Dop}(\mathcal{C}))$ of $\Fun(\Dop, \Fun(\Dop,
  \mathcal{C})) \simeq \Fun(\Dop \times \Dop, \mathcal{C})$,
  consisting of functors $\Dop \times \Dop \to \mathcal{C}$ that
  satisfy the Segal condition in each variable.
\end{defn}

\begin{defn}
  $\Seg_{\Dop}^{\mathcal{S}}(\CatI)$ denotes the full subcategory
  of $\Seg_{\Dop}(\CatI)$ consisting of Segal objects $X$ such that
  $X_{0}$ is an $\infty$-groupoid. We can then define
  $\Cat_{(\infty,2)}$ to be the full subcategory of
  $\Seg_{\Dop}^{\mathcal{S}}(\CatI)$ consisting of functors $X$ satisfying the
  \emph{completeness condition}, namely that the underlying Segal space
  $X^{\simeq}$ is complete in the sense of \cite{RezkCSS}.
\end{defn}

\begin{defn}
  $\Seg_{\Dop \times \Dop}^{\txt{2-fold}}(\mathcal{S})$ denotes
  the full subcategory of $\Seg_{\Dop \times \Dop}(\mathcal{S})$
  consisting of \emph{2-fold Segal spaces}, meaning those
  objects $X$ such that $X_{0,\bullet}\colon \Dop \to \mathcal{S}$ is
  constant. 
\end{defn}

\begin{remark}
  The top right vertical morphism in \cref{eq:segdiag} arises from the
  inclusion $\CatI \hookrightarrow \Seg_{\Dop}(\mathcal{S})$ of
  \icats{} as the complete Segal objects, due to
  Rezk~\cite{RezkCSS}. This also induces the other inclusions between
  the top two rows, and identifies $\CatIT$ with the full subcategory
  of $\Seg_{\Dop \times \Dop}^{\txt{2-fold}}(\mathcal{S})$ consisting
  of the complete 2-fold Segal spaces in the sense of
  Barwick~\cite{BarwickThesis}.
\end{remark}

\begin{defn}
  The category $\bTt$ has objects $[k](n_{1},\ldots,n_{k})$ for
  non-negative integers $k, n_{1},\ldots,n_{k}$, with a morphism
  $[k](n_{1},\ldots,n_{k}) \to [l](m_{1},\ldots,m_{l})$ given by a
  morphism $\phi \colon [k] \to [l]$ in $\simp$ together with a
  morphism $\psi_{ij} \colon [n_{i}] \to [m_{j}]$ in $\simp$ whenever
  $\phi(i-1) < j \leq \phi(i)$. Composition is defined in the obvious
  way, and we say this morphism is \emph{inert} or \emph{active} if $\phi$ and each of
  the maps $\psi_{ij}$ is inert or active, respectively.
\end{defn}

\begin{remark}
We can think of the objects of $\bTt$ as globular pasting diagrams, such as
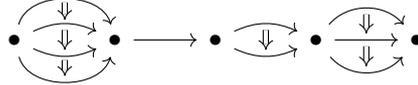
\[
\begin{tikzcd}
  \bullet 
  \ar[r, bend left=75, ""{below,name=A,inner sep=0.5pt}] 
  \ar[r, bend left=25, ""{name=B1,inner sep=0.5pt},
  ""{name=B2,below,inner sep=1pt}] 
  \ar[r, bend right=25, ""{name=C1,inner sep=1pt}, ""{below,name=C2,inner sep=0.5pt}] 
  \ar[r, bend right=75, ""{name=D,inner sep=1pt}] &
  \bullet 
  \ar[r]  &
  \bullet 
  \ar[r, bend left=25, ""{below,name=E,inner sep=1pt}] 
  \ar[r, bend right=25, ""{name=F,inner sep=1pt}] &
  \bullet 
  \ar[r, bend left=50, ""{below,name=G,inner sep=1pt}] 
  \ar[r, ""{name=H1,inner sep=1pt}, ""{name=H2,below,inner sep=1pt}] 
  \ar[r, bend right=50, ""{name=I,inner sep=1pt}] &
  \bullet
  \arrow[from=A,to=B1,Rightarrow]
  \arrow[from=B2,to=C1,Rightarrow]
  \arrow[from=C2,to=D,Rightarrow]
  \arrow[from=E,to=F,Rightarrow]
  \arrow[from=G,to=H1,Rightarrow]
  \arrow[from=H2,to=I,Rightarrow]
\end{tikzcd}
\]
which corresponds to the object $[4](3,0,1,2)$. This leads to the
equivalent definition of $\bTt$ as a full subcategory of the category of strict
2-categories, by thinking of the object $[k](n_1,\ldots,n_k)$ as the strict 2-category with objects
$0, \ldots, n$ whose category of morphisms $i \to j$ is
$\prod_{i < k \leq j} [n_{i}]$ if $i \leq j$ and empty otherwise, and with
composition given by taking products.
\end{remark}

\begin{notation}
We shall use the following special notation for the most basic objects of $\bbTheta_2$:
\begin{align*}
  C_0 &:= [0](),  \\
  C_1 &:= [1](0),  \\
  C_2 &:= [1](1).
\end{align*}
They can be pictured, respectively, as
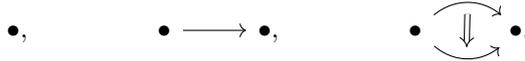
\[
\begin{tikzcd}
  \bullet,
  \end{tikzcd}
  \qquad\qquad
\begin{tikzcd}
  \bullet \ar[r] & \bullet,
  \end{tikzcd}
  \qquad\qquad
\begin{tikzcd}
  \bullet 
  \ar[r, bend left=40, ""{below,name=E,inner sep=2pt}] 
  \ar[r, bend right=40, ""{name=F,inner sep=2pt}] &
  \bullet . 
  \arrow[from=E,to=F,Rightarrow]
  \end{tikzcd}
\]
We refer to the object $C_{n}$ as the \emph{$n$-cell}; it is the
generic 2-category containing an $n$-morphism.
\end{notation}

\begin{defn}
  $\Seg_{\bTt^{\op}}(\mathcal{S})$ denotes the full subcategory of
  $\Fun(\bTt^{\op}, \mathcal{S})$ consisting of functors $X$ that
  satisfy the following pair of Segal conditions:
  \begin{itemize}
  \item for every object $[k](n_{1},\ldots,n_{k})$, the morphism
    \[ X([k](n_{1},\ldots,n_{k})) \to X([1](n_{1})) \times_{X(C_{0})}
      \cdots \times_{X(C_{0})} X([1](n_{k}))\]
    is an equivalence,
  \item for every object $[1](n)$, the morphism
    \[ X([1](n)) \to X(C_{2}) \times_{X(C_{1})} \cdots
      \times_{X(C_{1})} X(C_{2})\]
    is an equivalence.
  \end{itemize}
\end{defn}

\begin{remark}
  The bottom right vertical morphism in \cref{eq:segdiag} is given by composition with
  the functor $\tau \colon \simp \times \simp \to \bTt$, given on
  objects by $([k], [n]) \mapsto [k]([n],\ldots,[n])$. This restricts
  to an equivalence between $\Seg_{\bTt^{\op}}(\mathcal{S})$ and
  $\Seg_{\Dop \times \Dop}^{\txt{2-fold}}(\mathcal{S})$ and
  furthermore identifies $\CatIT$ with the full subcategory of
  \emph{complete} objects in $\Seg_{\bTt^{\op}}(\mathcal{S})$ in the
  sense of Rezk~\cite{RezkThetaN}; this comparison was first proved by
  Barwick and Schommer-Pries~\cite{BarwickSchommerPriesUnicity} and in
  different ways by Bergner and Rezk~\cite{BergnerRezk2} and the
  author \cite{thetan}.  
\end{remark}

\begin{notation} We introduce some notation for various structures
  related to \itcats{}:
  \begin{enumerate}[(i)]
  \item If $\mathcal{X}$ is an \itcat{}, we write
    $\iota_{1}\mathcal{X}$ for
    the underlying \icat{} of $\mathcal{X}$, and $\iota_{0}\mathcal{X}$
    for the underlying \igpd{}. If we view $\mathcal{X}$ as an object 
    $X_{\bullet} \in \Seg^{\mathcal{S}}_{\Dop}(\CatI)$, then $\iota_{1}\mathcal{X}$
    is the complete Segal space obtained by taking the underlying
    $\infty$-groupoid levelwise, \ie{} $X^{\simeq}_{\bullet}$, while
    $\iota_{0}\mathcal{X}$ is the $\infty$-groupoid $X_{0}$.
  \item If $\mathcal{X}$ is an \itcat{} and $x,y$ are objects of
    $\mathcal{X}$ then we write $\mathcal{X}(x,y)$ for the \icat{} of
    morphisms from $x$ to $y$ in $\mathcal{X}$. If we view $\mathcal{X}$
    as a simplicial \icat{} $X$, then this is given by the pullback
    square
    \[
      \begin{tikzcd}
        \mathcal{X}(x,y) \arrow{r} \arrow{d} & X_{1} \arrow{d} \\
        \{(x,y)\} \arrow{r} & X_{0} \times X_{0},
      \end{tikzcd}
      \]
    where the right vertical map is the functor induced by the two maps
    $[0] \to [1]$.
  \item If $\mathcal{X}$ and $\mathcal{Y}$ are \itcats{}, we write
    $\FUN(\mathcal{X}, \mathcal{Y})$ for the \itcat{} of functors
    between them, \ie{} the internal Hom in $\CatIT$, and
    $\Fun(\mathcal{X}, \mathcal{Y}) := \iota_{1}\FUN(\mathcal{X},
    \mathcal{Y})$ for its underlying \icat{}.
  \item If $\mathcal{X}$ is an \itcat{}, we write $\mathcal{X}^{\onop}$
    for the \itcat{} obtained from $\mathcal{X}$ by reversing the
    1-morphisms, and $\mathcal{X}^{\twop}$ for that obtained by
    reversing the 2-morphisms. If $\mathcal{X}$ is represented by a
    simplicial \icat{} $X_{\bullet}$ then $\mathcal{X}^{\twop}$
    corresponds to taking $\op$ levelwise to obtain $X_{\bullet}^{\op}$,
    while $\mathcal{X}^{\onop}$ is obtained by composing $X_{\bullet}$
    with the order-reversing involution of $\simp$.
  \end{enumerate}
\end{notation}

\begin{remark}
Another description of \itcats{} is that
they are precisely \icats{} \emph{enriched} in the symmetric monoidal
\icat{} $\CatI$, in the sense of \cite{enr}. This definition is shown
in \cite{enrcomp}
to be equivalent to \itcats{} viewed as complete objects in
$\Seg^{\mathcal{S}}_{\Dop}(\CatI)$, and hence is also equivalent to
the other definitions we have considered thus far; the comparison also
extends to an
equivalence between $\Seg_{\Dop}^{\mathcal{S}}(\CatI)$ and
\emph{categorical algebras} in $\CatI$, defined in \cite{enr} as
algebras for a family of (generalized non-symmetric) \iopds{}
$\Dop_{X}$. This allows us to construct certain \itcats{} as
\emph{free} algebras for these \iopds{}, as we will now explain:
\end{remark}

\begin{defn}
  A \emph{$\CatI$-graph} on a space $X$ is a functor $X \times X \to
  \CatI$; using the obvious naturality in $X$, these combine into an
  \icat{} $\txt{Graph}(\CatI)$. This can equivalently be viewed as the
  \icat{} $\Fun^{\mathcal{S}}(\simp^{\txt{el},\op}, \CatI)$
  where $\simp^{\txt{el}}$ is the subcategory of $\simp$ containing
  the objects $[0], [1]$ and the two inert maps $d_{0},d_{1} \colon
  [0] \to [1]$, and $\Fun^{\mathcal{S}}(\simp^{\txt{el},\op}, \CatI)$
  is the full subcategory of $\Fun(\simp^{\txt{el},\op}, \CatI)$
  consisting of functors $\Phi$ such that $\Phi_{0} \in
  \mathcal{S}$. The forgetful functor from categorical algebras to
  graphs then corresponds to the functor
  $\Seg_{\Dop}^{\mathcal{S}}(\CatI) \to \txt{Graph}(\CatI)$ induced by
  composition with the inclusion $\simp^{\txt{el}} \to \simp$. This
  has a left adjoint $\txt{Free} \colon \txt{Graph}(\CatI) \to
  \Seg_{\Dop}^{\mathcal{S}}(\CatI)$, which can be described by an
  explicit formula (as it is given by free algebras for a family of \iopds{}).
\end{defn}

\begin{defn}\label{defn:[n]free}
  In particular, given \icats{}
  $\mathcal{C}_{1},\ldots,\mathcal{C}_{n}$ we can define a
  $\CatI$-graph \[[n](\mathcal{C}_{1},\ldots,\mathcal{C}_{n})_{\txt{graph}}\] on the set $\{0,\ldots,n\}$ by
  \[ (i,j) \mapsto
    \begin{cases}
      \mathcal{C}_{j}, &  i = j-1,\\
      \emptyset, & \txt{otherwise}.
    \end{cases}
  \]
  We write $[n](\mathcal{C}_{1},\ldots,\mathcal{C}_{n})$ for the free
  \itcat{} on
  $[n](\mathcal{C}_{1},\ldots,\mathcal{C}_{n})_{\txt{graph}}$. The
  formula for free algebras implies that this \itcat{} has objects
  $0,\ldots,n$, and the \icats{} of maps are given by
  \[ [n](\mathcal{C}_{1},\ldots,\mathcal{C}_{n})(i,j) \simeq
    \begin{cases}
      \mathcal{C}_{i+1} \times \cdots \times \mathcal{C}_{j}, & i \leq
      j,\\
      \emptyset, & i > j,
    \end{cases}
  \]
  with composition given by the obvious equivalence
  \[ [n](\mathcal{C}_{1},\ldots,\mathcal{C}_{n})(i,j) \times
    [n](\mathcal{C}_{1},\ldots,\mathcal{C}_{n})(j,k) \isoto
    [n](\mathcal{C}_{1},\ldots,\mathcal{C}_{n})(i,k).\] Note that any
  inert map $\phi \colon [m] \to [n]$ in $\simp$ induces a fully
  faithful functor
  \[ \bar{\phi} \colon
    [m](\mathcal{C}_{\phi(1)},\ldots,\mathcal{C}_{\phi(m)}) \to
    [n](\mathcal{C}_{1},\ldots,\mathcal{C}_{n}),\] as the free functor
  on the inclusion of graphs determined by $\phi$.
\end{defn}

\begin{remark}
  In particular, we have a functor $[1](\blank) \colon \CatI \to
  \CatIT$ with two natural morphisms $[0] \to [1](\blank)$. From the
  free-forgetful adjunction for graphs, we see that for any \itcat{}
  $\mathcal{X}$ the fibre of
   \[ \Map([1](\mathcal{C}), \mathcal{X}) \to \Map([0], \mathcal{X})^{\times
       2} \]
   at objects $x,y \in \mathcal{X}$ is naturally equivalent to $\Map_{\CatI}(\mathcal{C},
   \mathcal{X}(x,y))$.
\end{remark}

We can use these free \itcats{} to describe some colimits of \itcats{}
that will be useful later on:
\begin{lemma}\label{lem:[n]Segal}
  For any \icats{} $\mathcal{C}_{1},\ldots,\mathcal{C}_{n}$, the
  functor
  \[ [1](\mathcal{C}_{1}) \amalg_{[0]} \cdots \amalg_{[0]}
    [1](\mathcal{C}_{n}) \to
    [n](\mathcal{C}_{1},\ldots,\mathcal{C}_{n}),\]
  induced by the inert maps $[0],[1] \to [n]$, is an equivalence.
\end{lemma}
\begin{proof}
  Since taking free \itcats{} is a left adjoint, this is the free
  functor on a morphism of graphs
  \[ [1](\mathcal{C}_{1})_{\txt{graph}} \amalg_{[0]_{\txt{graph}}}
    \cdots \amalg_{[0]_{\txt{graph}}}
    [1](\mathcal{C}_{n})_{\txt{graph}} \to
    [n](\mathcal{C}_{1},\ldots,\mathcal{C}_{n})_{\txt{graph}},\] which
  is obviously an equivalence.
\end{proof}

\begin{lemma}\label{lem:[1]wccolim}
  The functor $[1](\blank) \colon \CatI \to \CatIT$ preserves weakly contractible colimits.
\end{lemma}
\begin{proof}
  Given a diagram $f \colon \mathcal{I} \to \CatI$ and
  $\mathcal{X} \in \CatIT$, we have a
  natural commutative square
  \[
    \begin{tikzcd}
      \Map([1](\colim_{\mathcal{I}} f), \mathcal{X}) \arrow{r} \arrow{d} &
      \lim_{\mathcal{I}^{\op}} \Map([1](f), \mathcal{X}) \arrow{d} \\
      \Map([0], \mathcal{X})^{\times 2} \arrow{r} &
      \lim_{\mathcal{I}^{\op}} \Map([0], \mathcal{X})^{\times 2}.
    \end{tikzcd}
  \]
  If $\mathcal{I}$ is weakly contractible then the bottom horizontal
  morphism is an equivalence, so to show the top horizontal morphism
  is an equivalence it suffices to show it is an equivalence on the
  fibre at any pair of objects $x,y \in \mathcal{C}$. Since limits
  commute, we can identify the map on
  fibres as
  \[ \Map_{\CatI}(\colim_{\mathcal{I}} f, \mathcal{X}(x,y)) \to
    \lim_{\mathcal{I}^{\op}} \Map_{\CatI}(f,
    \mathcal{X}(x,y)),\]
  which is indeed an equivalence.
\end{proof}

\begin{lemma}\label{lem:[1]wcrep}
  If $\mathcal{I}$ is a weakly contractible \icat{} and $\mathcal{X}$
  is an \itcat{} which corresponds to a simplicial \icat{} $X_{\bullet}$, then there is a
  natural equivalence
  \[ \Map_{\CatIT}([1](\mathcal{I}), \mathcal{X}) \simeq
    \Map_{\CatI}(\mathcal{I}, X_{1}).\]
\end{lemma}
\begin{proof}
  We have a natural fibre sequence
  \[ \Map(\mathcal{I}, \mathcal{X}(x,y)) \to \Map(\mathcal{I},
    X_{1}) \to \Map(\mathcal{I}, X_{0})^{\times 2}
    \simeq \Map(\| \mathcal{I} \|, X_{0})^{\times 2},\]
  where the equivalence uses that $X_{0}$ is an $\infty$-groupoid.
  If $\mathcal{I}$ is weakly contractible, this is equivalent to the
  fibre sequence above for $\Map([1](\mathcal{I}), \mathcal{X})$.
\end{proof}

\section{The Gray Tensor Product}\label{sec:gray}
For ordinary (strict) 2-categories, Gray~\cite{GrayFormal} defined a
(non-symmetric) tensor product $\otimes^{\pcolax}$ (where
$\mathbf{A} \otimes^{\lax} \mathbf{B} \cong \mathbf{B}
\otimes^{\colax} \mathbf{A}$), colimit-preserving in each variable,
such that the internal Homs are 2-categories of functors where
morphisms are either lax or colax\footnote{We have tried to follow the
  convention that the prefix ``co'' refers to reversing the direction
  of 2-morphisms, while ``op'' refers to reversing that of
  1-morphisms. Since the two types of lax natural transformations are
  related by reversing 2-morphisms, we call them lax and colax
  transformations (just as we would refer to lax and colax functors,
  though these do not appear in this paper). However, in the
  2-categorical literature the term \emph{oplax} natural
  transformation is also common.}  natural transformations (depending
on whether we take the adjoint in the first or second variable). In
this section we first recall an explicit description of
$I \otimes^{\colax}J$ for $I,J \in \bTt$ and then discuss the
(expected) extension of the Gray tensor product to \itcats{} and its
basic properties.

\begin{notation}
Recall that $[k]$ denotes the ordered set $\{0 < 1 < \cdots < n\}$.
Viewing this as a poset, the product $[k]\times [m]$
of posets has the shape of a rectangular grid.
This is a ranked poset; its maximal chains (\ie{} the paths from
$(0,0)$ to $(k,m)$)
all have length $k+m$ and form a poset denoted $\operatorname{MaxCh}([k]{\times}[m])$,
whose partial order relation is generated by 
\[
\rotatebox[origin=c]{180}{$\Lsh$} \leq \rotatebox[origin=c]{270}{$\Rsh$}
\]
(Note that this poset is isomorphic to the poset $\operatorname{Sh}(k,m)$ of 
\emph{$(k,m)$-shuffles}, ordered with $k+m$ as the least element and $m+k$ as the 
greatest element.)
\end{notation}

\begin{notation}
  For non-negative integers $i \leq j$ it is convenient to also
  introduce the notation $[i,j]$ for the ordered set $\{i < i+1 < \cdots
  < j \}$, which is isomorphic to $[j-i]$. If $i > j$ it is convenient
  to take $[i,j] = \emptyset$.
\end{notation}

\begin{defn}\label{defn:GrayTh2}
  If $I = [n](x_1,\ldots,x_n)$ and $J = [m](y_1,\ldots,y_m)$ are
  objects of $\bTt$, then the Gray tensor product
  $I \otimes^{\colax} J$ is the
  $2$-category with object set $\obj([n]) \times \obj([m])$ and
  Hom-categories (actually posets)
  \[
    \begin{split}
  \Hom((i,j), (i',j')) & :=
  \begin{cases}
    \operatorname{MaxCh}([i,i']{\times}[j,j']) \times 
\prod_{i < s \leq i'} [x_s] \times \prod_{j < t \leq j'} [y_t], & i\leq
i',j \leq j' \\
    \emptyset, & \txt{otherwise}
  \end{cases}
  \\ &
  \cong
    \operatorname{MaxCh}([i,i']{\times}[j,j']) \times 
I(i,i') \times J(j,j')
    \end{split}
  \]
The composition of morphisms $(i,j) \to (i',j')$ and
$(i',j') \to (i'',j'')$ is defined by combining the composition in $I$
and $J$ with the natural inclusion
\[ \operatorname{MaxCh}([i,i']{\times}[j,j']) \times
  \operatorname{MaxCh}([i',i'']{\times}[j',j'']) \to
  \operatorname{MaxCh}([i,i'']{\times}[j,j''])\]
that combines a path from $(i,j)$ to $(i',j')$ with a path from
$(i',j')$ to $(i'',j'')$ to get the subset of paths from $(i,j)$ to
$(i'',j'')$ that factor through $(i',j')$. With this definition there
is also a canonical way to define 
functors between Gray tensor products from morphisms in $\bTt$, so
that we obtain a functor $\otimes^{\colax} \colon \bTt \times \bTt \to \CatIT$.
\end{defn}

\begin{exs}
$[1](0) \otimes^{\colax} [1](0) $ has $4$ objects, $00, 01, 10, 11$,
and 
\(
\Hom(00,11) = (\rotatebox[origin=c]{180}{$\Lsh$} \leq 
\rotatebox[origin=c]{270}{$\Rsh$})
\).
The remaining hom categories are discrete: contractible if the indices are 
non-decreasing, empty if some index decreases.
The whole $2$-category can therefore be depicted as a colax square:
\[
\begin{tikzcd}
  00 \ar[r] \ar[d] & 01 \ar[d] \\
  10 \ar[r] \ar[ru, Rightarrow] & 11
\end{tikzcd}
\]
Similarly, $C_2 \otimes^{\colax} C_{1} = [1](1) \otimes^{\colax} [1](0)$
has the shape
of a cylinder (with side squares colax):
\[
\begin{tikzcd}
  00 \ar[rr] \ar[d, bend right=45, ""{right,name=A1}] && 01 \ar[d,
  bend right=45, ""{right,name=A2}] \ar[d, bend left=45, ""{left,name=B2}] \\
  10 \ar[rr] 
  \ar[rru, bend right=30, 
  start anchor=north east, end anchor=south, crossing over, Rightarrow]
  \ar[rru, bend left=30, 
   start anchor=north, end anchor=south west,Rightarrow]
  \ar[u, leftarrow, crossing over, bend right=45, ""{left,name=B1}]
  && 11.
  \arrow[from=A1,to=B1,Rightarrow]
  \arrow[from=A2,to=B2,Rightarrow]
\end{tikzcd}
\]
This means that a diagram of shape $C_{2} \otimes^{\colax} C_{1}$ in
an \itcat{} $\mathcal{X}$ consists of the following data in $\mathcal{X}$:
\begin{itemize}
\item objects $X,Y,X',Y'$,
\item morphisms $f,g \colon X \to Y$, $f',g' \colon X' \to Y'$, $\xi
  \colon X \to X'$, $\eta \colon Y \to Y'$,
\item 2-morphisms $\alpha \colon f \to g$, $\alpha' \colon f' \to g'$,
  $\phi \colon \eta f \to f' \xi$, $\psi \colon \eta g \to g' \xi$,
\item an equivalence 
  $\psi \circ (\eta \alpha) \simeq (\alpha' \xi) \circ \phi$
  of 2-morphisms $\eta f \to g' \xi$.
\end{itemize}
\end{exs}

Since we can view 2-categories as \itcats{}, the classical Gray tensor
product induces a functor
\[ \otimes^{\colax} \colon \bbTheta_{2} \times \bbTheta_{2} \to
  \CatIT.\]
We will make the following three assumptions about this functor:
\begin{assumption}\label{ass:gray}\ 
  \begin{enumerate}[(1)]
  \item The functor $\otimes^{\colax}$ satisfies the co-Segal condition\footnote{By the co-Segal condition for a functor
      $\phi \colon \bTt \to \mathcal{C}$ we mean the Segal condition
      for $\phi^{\op} \colon \bTt^{\op}\to \mathcal{C}^{\op}$.}
    in each variable. The
    unique extension to a functor $\mathcal{P}(\bTt) \times
    \mathcal{P}(\bTt) \to \CatIT$ that preserves colimits in each
    variable therefore uniquely factors through a functor
    $\Seg_{\bTt^{\op}}(\mathcal{S}) \times
    \Seg_{\bTt^{\op}}(\mathcal{S}) \to \CatIT$ that preserves colimits
    in each variable.
  \item The functor $\Seg_{\bTt^{\op}}(\mathcal{S}) \times
    \Seg_{\bTt^{\op}}(\mathcal{S}) \to \CatIT$ takes fully faithful
    and essentially surjective morphisms in each variable to
    equivalences, and thus factors uniquely through a functor
    \[ \otimes^{\lax} \colon \CatIT \times \CatIT \to \CatIT.\]
  \item The restriction of $\otimes^{\lax}$ to ordinary (strict)
    2-categories agrees with the classical Gray tensor
    product.\footnote{In fact, we only need this assumption in the
      case of \emph{gaunt} 2-categories, meaning ones with no
      non-trivial invertible 1- or 2-morphisms, which may be more
      straightforward to prove than the general case.}
  \end{enumerate}
\end{assumption}

\begin{remark}\label{rmk:grayconstr}
  Assumptions (1) and (2) have recently been proved by Y.~Maehara
  \cite{MaeharaGray}, who shows that formally extending the ordinary
  Gray tensor product on $\bTt$ gives a left Quillen bifunctor for
  $\bTt$-sets. Several other constructions of Gray tensor products in
  various models of $(\infty,2)$-categories (some more generally in
  $(\infty,n)$-categories) have also recently appeared, including
  \cite{GagnaHarpazLanariGray,OzornovaRovelliVerityGray,CKMComical}.
\end{remark}

\begin{remark}
  As observed by Ayala--Francis~\cite{AyalaFrancisFlagged}, a colimit
  diagram in $\CatIT$ whose underlying diagrams of \icats{} and
  \igpds{} are also colimit diagrams is a colimit in
  $\Seg_{\bbTheta_{2}^{\op}}(\mathcal{S})$. This is true for the
  diagrams exhibiting the co-Segal condition for $\otimes^{\colax}$,
  hence we can also take left Kan extensions to obtain a functor
\[ \otimes^{\colax}_{\Seg} \colon \Seg_{\bbTheta_{2}^{\op}}(\mathcal{S}) \times
  \Seg_{\bbTheta_{2}^{\op}}(\mathcal{S}) \to
  \Seg_{\bbTheta_{2}^{\op}}(\mathcal{S}),\]
colimit-preserving in each variable, such that there is a commutative
diagram
\[
  \begin{tikzcd}
    \Seg_{\bbTheta_{2}^{\op}}(\mathcal{S}) \times
    \Seg_{\bbTheta_{2}^{\op}}(\mathcal{S}) \arrow{d}
    \arrow{r}{\otimes^{\colax}_{\Seg}} &
    \Seg_{\bbTheta_{2}^{\op}}(\mathcal{S}) \arrow{d} \\
    \CatIT \times \CatIT \arrow{r}{\otimes^{\colax}} & \CatIT,
  \end{tikzcd}
\]
where the vertical morphisms are given by localization.
\end{remark}

\begin{defn}\label{defn:laxcolaxgray}
  For \itcats{} $\mathcal{X}$ and $\mathcal{Y}$ we call
  $\mathcal{X} \otimes^{\colax} \mathcal{Y}$ the \emph{colax Gray tensor
    product} of $\mathcal{X}$ and $\mathcal{Y}$. We will also write
  $\mathcal{X} \otimes^{\lax} \mathcal{Y} := \mathcal{Y}
  \otimes^{\colax} \mathcal{X}$, and call this the {\em lax Gray
    tensor product}.
\end{defn}

\begin{defn}
  The functor $\otimes^{\pcolax}$ preserves colimits in each variable,
  and so has adjoints
  $\FUN(\blank,\blank)_{\txt{(co)lax}}$, which satisfy
  \[
    \begin{split}
    \Map_{\CatIT}(\mathcal{X}, \FUN(\mathcal{Y}, \mathcal{Z})_{\colax}) 
	&\;\simeq\;
	\Map_{\CatIT}(\mathcal{Y} \otimes^{\colax} \mathcal{X}, \mathcal{Z})
	\\
	&\;\simeq\;
    \Map_{\CatIT}(\mathcal{X} \otimes^{\lax} \mathcal{Y}, \mathcal{Z})
	\;\simeq\; \Map_{\CatIT}(\mathcal{Y}, \FUN(\mathcal{X},\mathcal{Z})_{\lax}) 
	.  
    \end{split}
  \]
  A \emph{(co)lax natural
    transformation} (between functors $\mathcal{X} \to \mathcal{Y}$)
	is a functor of $(\infty,2)$-categories
  \[ \mathcal{X} \otimes^{\txt{(co)lax}} \Delta^{1} \to \mathcal{Y}.\]
  The $(\infty,2)$-category $\FUN(\mathcal{X},
  \mathcal{Y})_{\txt{(co)lax}}$ thus has usual functors of
  $(\infty,2)$-categories as objects, and (co)lax natural
  transformations as morphisms. Similarly, the $2$-morphisms are functors of 
  $(\infty,2)$-categories $\mathcal{X} \otimes^{\txt{(co)lax}} C_2 
  \to \mathcal{Y}$.
\end{defn}

\begin{remark}
  A lax natural transformation $\eta$ between functors
  $F,G \colon \mathcal{X} \to \mathcal{Y}$ assigns to every morphism
  $f \colon X \to X'$ in $\mathcal{X}$ a lax square
  \[
    \begin{tikzcd}
      F(X) \arrow{r}{\eta_{X}} \arrow{d}[swap]{F(f)} & G(X) \arrow{d}{G(f)}
      \arrow[dl,Rightarrow,inner sep=2pt]\\
      F(X') \arrow{r}{\eta_{X'}} & G(X'),
    \end{tikzcd}
  \]
  while a colax natural transformation assigns a colax square
    \[
    \begin{tikzcd}
      F(X) \arrow{r}{\eta_{X}} \arrow{d}[swap]{F(f)} & G(X) \arrow{d}{G(f)}
      \arrow[dl,Leftarrow,inner sep=2pt]\\
      F(X') \arrow{r}{\eta_{X'}} & G(X').
    \end{tikzcd}
  \]
\end{remark}

\begin{remark}\label{rmk:colaxigpd}
  Note that if $f$ in the previous remark is $\id_{X}$ then our
  definition requires the (co)lax square to be the identity of
  $\eta_{X}$, and then the compatibility with composition implies that
  if $f$ is an equivalence then the (co)lax square commutes. This
  suggests that for functors from an \igpd{} (co)lax natural
  transformations should reduce to ordinary natural
  transformations. To see this more formally, first note that if $X$
  is an \igpd{} then the natural equivalence
  $X \simeq \colim_{X} C_{0}$ induces for any \itcat{} $\mathcal{Y}$
  an equivalence
  \[ X \otimes^{\colax} \mathcal{Y} \simeq \colim_{X} (C_{0}
    \otimes^{\colax} \mathcal{Y}) \simeq \colim_{X} \mathcal{Y} \simeq
    X \times \mathcal{Y}.\]
  Hence if $\mathcal{Z}$ is another \itcat{}, we have natural
  equivalences
  \[ \Map(\mathcal{Y}, \FUN(X,\mathcal{Z})_{\colax}) \simeq \Map(X
    \otimes^{\colax} \mathcal{Y}, \mathcal{Z}) \simeq \Map(X \times
    \mathcal{Y}, \mathcal{Z}) \simeq \Map(\mathcal{Y},
    \FUN(X,\mathcal{Z})),\]
  which implies by the Yoneda lemma that we indeed have a natural equivalence
  \[ \FUN(X,\mathcal{Z})_{\colax} \simeq \FUN(X,\mathcal{Z}). \]
\end{remark}

\begin{remark}
  The paper \cite{JohnsonFreydScheimbauerLax} of Johnson-Freyd and
  Scheimbauer gives an alternative construction of the \itcats{} of
  $\FUN(\mathcal{X},\mathcal{Y})_{\pcolax}$, without defining the Gray
  tensor product in general: in our notation they give explicit
  definitions of \itcats{} corresponding to our $\FUN(\tau([n],[m]),
  \mathcal{Y})_{\pcolax}$ using the functor $\tau \colon \simp \times
  \simp \to \bTt$, and then define
  $\FUN(\mathcal{X},\mathcal{Y})_{\lax}$ as the 2-fold Segal space
  \[\Map(\mathcal{X}, \FUN(\tau(\blank,\blank), \mathcal{Y})_{\colax}).\]
\end{remark}

\begin{propn}\label{propn:gray2op}
  There is a natural equivalence
  \[ (\mathcal{X} \otimes^{\lax} \mathcal{Y})^{\txt{2-op}} \simeq
    \mathcal{X}^{\txt{2-op}} \otimes^{\colax} \mathcal{Y}^{\txt{2-op}}.\]
\end{propn}
\begin{proof}
  There is such an equivalence for the tensor product of ordinary
  2-categories, so there is a natural equivalence for  $\mathcal{X},
  \mathcal{Y} \in \bbTheta_{2}$, which extends by colimits to an
  equivalence for all $\mathcal{X},\mathcal{Y}$.
\end{proof}
\begin{cor}\label{cor:FUNlax2op}
  There is a natural equivalence
  \[ \FUN(\mathcal{X}, \mathcal{Y})_{\lax}^{\twop} \simeq
    \FUN(\mathcal{X}^{\twop}, \mathcal{Y}^{\twop})_{\colax}\]
  for all \itcats{} $\mathcal{X}, \mathcal{Y}$.\qed
\end{cor}

\begin{remark}\label{rmk:graytocart}
  Since $C_{0} \otimes^{\colax} \blank \simeq \id \simeq \blank
  \otimes^{\colax} C_{0}$, we obtain natural morphisms $\mathcal{X}
  \otimes^{\colax} \mathcal{Y} \to \mathcal{X} \otimes^{\colax} C_{0}
  \simeq \mathcal{X}$ and $\mathcal{X}
  \otimes^{\colax} \mathcal{Y} \to C_{0} \otimes^{\colax} \mathcal{Y}
  \simeq \mathcal{Y}$, and so
  a natural morphism \[\mathcal{X}
    \otimes^{\colax} \mathcal{Y} \to \mathcal{X} \times \mathcal{Y}.\]
  We
will now observe that this exhibits $\mathcal{X} \times \mathcal{Y}$
as a localization of $\mathcal{X}
\otimes^{\colax} \mathcal{Y}$:  
\end{remark}

\begin{propn}\label{propn:laxtocart}
  \ 
  \begin{enumerate}[(i)]
  \item If $\mathcal{C}$ and $\mathcal{D}$ are \icats{}, the functor
    $\mathcal{C} \otimes^{\colax} \mathcal{D} \to \mathcal{C} \times
    \mathcal{D}$ exhibits $\mathcal{C} \times \mathcal{D}$ as the
    \icat{} $L_{(\infty,1)}(\mathcal{C} \otimes^{\colax} \mathcal{D})$
    obtained by inverting all 2-morphisms in $\mathcal{C}
    \otimes^{\colax} \mathcal{D}$.
  \item The natural commutative square
    \[
      \begin{tikzcd}
        \iota_{1}\mathcal{X} \otimes^{\colax} \iota_{1}\mathcal{Y}
        \arrow{d} \arrow{r} & \iota_{1}\mathcal{X} \times
        \iota_{1}\mathcal{Y} \arrow{d} \\
        \mathcal{X} \otimes^{\colax} \mathcal{Y} \arrow{r} & \mathcal{X}
        \times \mathcal{Y}
      \end{tikzcd}
    \]
    is a pushout square for all \itcats{} $\mathcal{X}, \mathcal{Y}$.
  \end{enumerate}
\end{propn}
\begin{proof}
  To prove (i) it suffices, since both sides preserve colimits in
  each variable,  to show that this morphism is an equivalence for
  $\mathcal{C}$ and $\mathcal{D}$ either $C_{0}$ or $C_{1}$. The only
  non-trivial case is $C_{1} \otimes^{\colax} C_{1} \to C_{1} \times
  C_{1}$, which indeed exhibits the commuting square $C_{1} \times
  C_{1}$ as obtained by inverting the unique 2-morphism in $C_{1}
  \otimes^{\colax} C_{1}$.
  
  To prove (ii), it suffices to prove the analogue of (ii) for the pairing $\otimes^{\colax}_{\Seg}$ on
  $\Seg_{\bbTheta_{2}^{\op}}(\mathcal{S})$, from which
  $\otimes^{\colax}$ is obtained by localization.
  This also preserves
  colimits in each variable, and $\iota_{1}$ on
  $\Seg_{\bbTheta_{2}^{\op}}(\mathcal{S})$ preserves colimits, so it
  suffices to check the square is a pushout for $\mathcal{X},
  \mathcal{Y}$ being either $C_{0}$, $C_{1}$ or $C_{2}$. Here it
  follows from the
  description of the Gray tensor product in \cref{defn:GrayTh2} that 
  \[
    \begin{tikzcd}
      \iota_{1}C_{i} \otimes^{\colax} \iota_{1}C_{j}
      \arrow{d} \arrow{r} & \iota_{1}C_{i} \times
      \iota_{1}C_{j} \arrow{d} \\
      C_{i} \otimes^{\colax} C_{j} \arrow{r} & C_{i}
      \times C_{j}
    \end{tikzcd}
  \]
  is a pushout square in $\Fun(\Dop, \mathcal{S})$ and hence in the
  localization $\CatIT$ since these are already local objects.
\end{proof}

Composing with the natural map from \cref{rmk:graytocart} we get for
any \itcats{} $\mathcal{X},\mathcal{Y},\mathcal{Z}$ a natural map
\[\Map_{\CatIT}(\mathcal{X} \times \mathcal{Z}, \mathcal{Y}) \to
  \Map_{\CatIT}(\mathcal{X}, \otimes^{\colax} \mathcal{Z},
  \mathcal{Y}),\]
which by adjunction induces a natural map
\[ \FUN(\mathcal{X}, \mathcal{Y}) \to
  \FUN(\mathcal{X},\mathcal{Y})_{\colax}.\]
We will now show that this identifies $\FUN(\mathcal{X},
\mathcal{Y})$ with a subobject of $\FUN(\mathcal{X},\mathcal{Y})_{\pcolax}$:
\begin{cor}
  There is a natural identification of $\FUN(\mathcal{X},
  \mathcal{Y})$ with the sub-\itcat{} of
  $\FUN(\mathcal{X}, \mathcal{Y})_{\pcolax}$ containing all objects,
  with  1-morphisms the
  (co)lax natural transformations all of whose (co)lax naturality
  squares commute, and all 2-morphisms between these.
\end{cor}
\begin{proof}
  Let $\mathcal{X}$ and $\mathcal{Y}$ be \itcats{}, and consider the
  commutative diagram
  \[
    \begin{tikzcd}
      \iota_{1}\mathcal{X} \otimes^{\colax} \iota_{1}\mathcal{Y}
      \arrow{r} \arrow{d} & \iota_{1}\mathcal{X} \times
      \iota_{1}\mathcal{Y} \arrow{d} \\
      \iota_{1}\mathcal{X} \otimes^{\colax}\mathcal{Y} \arrow{r}
      \arrow{d} & \iota_{1}\mathcal{X} \times \mathcal{Y} \arrow{d} \\
      \mathcal{X} \otimes^{\colax} \mathcal{Y} \arrow{r} & \mathcal{X}
      \times \mathcal{Y}.
    \end{tikzcd}
  \]
  Here the top square and the outer square are pushouts by
  \cref{propn:laxtocart}, hence so is the bottom
  square. Given a third \itcat{} $\mathcal{Z}$ we obtain a commutative
  diagram 
  \[
    \begin{tikzcd}
    \Map(\mathcal{X} \times \mathcal{Y}, \mathcal{Z}) \arrow{r}
    \arrow{d} & \Map(\mathcal{X} \otimes^{\colax} \mathcal{Y},
    \mathcal{Z}) \arrow{d} \\
    \Map(\iota_{1}\mathcal{X} \times \mathcal{Y}, \mathcal{Z})
    \arrow{d} \arrow{r} & \Map(\iota_{1}\mathcal{X} \otimes^{\colax} \mathcal{Y},
    \mathcal{Z}) \arrow{d} \\
    \Map(\iota_{1} \mathcal{X} \times \iota_{1}\mathcal{Y},
    \mathcal{Z}) \arrow{r} & \Map(\iota_{1}\mathcal{X} \otimes^{\colax}
    \iota_{1}\mathcal{Y}, \mathcal{Z}),
  \end{tikzcd}
\]
  where all squares are cartesian. We can rewrite this as
  \[
    \begin{tikzcd}
    \Map(\mathcal{X}, \FUN(\mathcal{Y}, \mathcal{Z})) \arrow{r} \arrow{d} &
    \Map(\mathcal{X}, \FUN(\mathcal{Y}, \mathcal{Z})_{\lax})
    \arrow{d} \\
        \Map(\iota_{1}\mathcal{X}, \FUN(\mathcal{Y}, \mathcal{Z})) \arrow{r} \arrow{d}&
    \Map(\iota_{1}\mathcal{X}, \FUN(\mathcal{Y},
    \mathcal{Z})_{\lax}) \arrow{d} \\
    \Map(\iota_{1}\mathcal{X}, \FUN(\iota_{1}\mathcal{Y}, \mathcal{Z})
    \arrow{r} & \Map(\iota_{1}\mathcal{X}, \FUN(\iota_{1}\mathcal{Y},
    \mathcal{Z})_{\lax}).
  \end{tikzcd}
\]
This says, firstly, that a functor $\mathcal{X} \to \FUN(\mathcal{Y},
\mathcal{Z})_{\lax}$ factors through $\FUN(\mathcal{Y},
\mathcal{Z})$ \IFF{} its restriction to the underlying \icat{}
$\iota_{1}\mathcal{X}$ does so. In other words, $\FUN(\mathcal{Y},
\mathcal{Z}) \to \FUN(\mathcal{Y},\mathcal{Z})_{\lax}$ is locally
fully faithful. Furthermore, a functor from $\iota_{1}\mathcal{X}$
factors through $\FUN(\mathcal{Y}, \mathcal{Z})$ \IFF{} the induced
functor to $\FUN(\iota_{1}\mathcal{Y}, \mathcal{Z})_{\lax}$ factors
through $\FUN(\iota_{1}\mathcal{Y}, \mathcal{Z})$, which we can
interpret via \cref{propn:laxtocart}(i) as saying that the
adjoint functor $\iota_{1}\mathcal{X} \otimes^{\colax}
\iota_{1}\mathcal{Y} \to \mathcal{Z}$ takes all 2-morphisms in
$\iota_{1}\mathcal{X} \otimes^{\colax} \iota_{1}\mathcal{Y}$ to
equivalences in $\mathcal{Z}$, as required.
\end{proof}
Finally, we note the following colimit decomposition of the Gray
tensor product of the generators $C_{i}$:
\begin{lemma}\label{lem:graytenscolim}
  We have the following colimit decompositions in
  $\Seg_{\bbTheta_{2}^{\op}}(\mathcal{S})$ (and hence in $\CatIT$):  
  \[ C_{1}\otimes^{\colax} C_{1} \simeq {[2]}(0,0) \cup_{C_{1}} C_{2}
    \cup_{C_{1}} [2](0,0),\]
  \[ C_{2}\otimes^{\colax} C_{1} \simeq {[2]}(1,0) \cup_{C_{2}}
    [1]([1]^{2}) \cup_{C_{2}} {[2]}(0,1),\]
  \[ C_{2}\otimes^{\colax} C_{2} \simeq {[2]}(1,1) \cup_{[1]([1]^{2})} [1]([1]^{3})
    \cup_{[1]([1]^{2})} {[2]}(1,1).\]
  where the maps in the colimits are the obvious ones.
\end{lemma}
\begin{proof}
  It suffices to prove that these give colimit diagrams in
  $\Seg_{\bbTheta_{2}^{\op}}(\mathcal{S})$. But in fact in all three
  cases it is easy to see that we have a colimit diagram already in
  the \icat{} $\Fun(\bbTheta_{2}^{\op},\mathcal{S})$ of presheaves.
\end{proof}

\section{Lax Morphisms of Adjunctions}\label{sec:adj}
In this section we will study (co)lax morphisms of adjunctions in an
\itcat{}, which arise as a special case of (co)lax natural
transformations:
\begin{notation}
  Let $\fadj$ denote the ``walking adjunction'' 2-category, i.e.~the
  free 2-category containing an adjunction. Following
  \cite{SchanuelStreet}, Riehl and
  Verity~\cite{RiehlVerityAdj} give a combinatorial description of
  this 2-category; we will not recall this here, but for notational
  convenience we will name the lower-dimensional parts of the
  category: it has two objects, $-$ and $+$, and morphisms are
  generated by $l \colon - \to +$ (the left adjoint) and
  $r \colon + \to -$ (the right adjoint).
\end{notation}
\begin{defn}
  Let $\mathcal{X}$ be an \itcat{}. An \emph{adjunction} in
  $\mathcal{X}$ is a functor of \itcats{} $\fadj \to \mathcal{X}$, and
  a \emph{(co)lax morphism} of adjunctions is a (co)lax natural
  transformation between adjunctions, \ie{} a functor
  \[ \fadj \otimes^{\pcolax} C_{1} \to \mathcal{X}.\] We write
  $\ADJ(\mathcal{X})_{\pcolax} := \FUN(\fadj, \mathcal{X})_{\pcolax}$
  for the \itcat{} of adjunctions in $\mathcal{X}$ and (co)lax
  morphisms between them, and $\Adj(\mathcal{X})_{\pcolax}$ for the
  underlying \icat{}.
\end{defn}

\begin{remark}\label{rmk:adjsymm}
The symmetry of the definition of $\fadj$ gives
equivalences
\begin{itemize}
\item $\fadj^{\twop} \simeq \fadj$, interchanging $-$ and $+$ and
  swapping $l$ and $r$,  
\item $\fadj^{\op} \simeq \fadj$, fixing the objects but interchanging $l$ and $r$,
\end{itemize}
Combined with \cref{cor:FUNlax2op}, the first gives a natural
equivalence
\[ \ADJ(\mathcal{X})_{\lax}^{\twop} \simeq
  \ADJ(\mathcal{X}^{\twop})_{\colax}.\]
\end{remark}

For ordinary 2-categories, one can show that
\begin{itemize}
\item a lax morphism of adjunctions corresponds to
  a commutative square of right adjoints,
\item a colax morphism of adjunctions corresponds to
  a commutative square of left adjoints.
\end{itemize}
Our goal in this section is to extend these equivalences to the \itcatl{}
setting, \ie{} to identify the \itcats{} $\ADJ(\mathcal{X})_{\pcolax}$
with the full subcategories of the arrow \itcat{} $\FUN(C_{1},
\mathcal{X})$ spanned by the morphisms that are right and left
adjoints, respectively. Our starting point is the following result of Riehl and Verity:
\begin{thm}[Riehl--Verity \cite{RiehlVerityAdj}]\label{thm:rvadj}
  Let $\mathcal{X}$ be an $(\infty,2)$-category, and denote by
 $\Map(\Delta^{1}, \mathcal{X})^{\ladj}$ and $\Map(\Delta^{1},
 \mathcal{X})^{\radj}$ the subspaces of $\Map(\Delta^{1},
 \mathcal{X})$ consisting of those components that correspond to left and right
 adjoint 1-morphisms, respectively. Then the maps
 \[\Map(\mathfrak{adj}, \mathcal{X}) \to \Map(\Delta^{1},
    \mathcal{X})^{\ladj}, \quad \Map(\mathfrak{adj}, \mathcal{X}) \to \Map(\Delta^{1},
    \mathcal{X})^{\radj}\] given by evaluation at the morphisms $l$ and
  $r$, respectively, are both equivalences.
\end{thm}
Our description of the \itcats{} $\ADJ(\mathcal{X})_{\pcolax}$ will
follow from a description of certain adjoints in \itcats{} of the form
$\FUN(\mathcal{Y}, \mathcal{X})_{\pcolax}$. To state this we need some terminology:
\begin{remark}\label{rmk:mate}
  Given a colax square
  \begin{equation}
    \label{eq:colaxsq}
    \begin{tikzcd}
      A \arrow[swap]{d}{a} \arrow{r}{l} & B
      \arrow{d}{b} \arrow[Leftarrow, shorten >= 6pt,shorten <= 6pt]{dl}{\phi}\\
      A' \arrow[swap]{r}{l'} & B'
    \end{tikzcd}
  \end{equation}
in some \itcat{}, where $\phi$ is a 2-morphism $l'a \to bl$
and the morphisms $l$ and $l'$ have right adjoints $r$ and $r'$,
respectively, then the \emph{mate} of $\phi$ is the 
transformation $ar \to r'b$ given by the composite
\[ ar \to r'l'ar \xto{r' \phi r} r'blr \to r'b,\]
using the unit $\id \to r'l'$ and the counit $lr \to \id$. We can
depict this as a \emph{lax} square
\begin{equation}
  \label{eq:laxsq}
\begin{tikzcd}
  B \arrow[swap]{d}{b} \arrow{r}{r} & A
  \arrow{d}{a} \arrow[Rightarrow, shorten >= 6pt,shorten <= 6pt]{dl}\\
 B' \arrow[swap]{r}{r'} & B,
\end{tikzcd}
\end{equation}
Similarly, given a lax square \cref{eq:laxsq} whose horizontal
morphisms are right adjoints, we can produce a colax mate square
\cref{eq:colaxsq}. The adjunction identities moreover imply that
taking mates twice gives back the original square.
\end{remark}

\begin{thm}\label{propn:laxfunadj}
  Let $\mathcal{X}$ be an \itcat{}.
  \begin{enumerate}[(i)]
  \item A 1-morphism in $\FUN(\mathcal{Y}, \mathcal{X})_{\lax}$, \ie{} a lax natural
    transformation $\phi \colon F \to G$, is a left adjoint
    \IFF{} 
    for every $Y \in \mathcal{Y}$ the morphism $\phi_{Y} \colon F(Y)
    \to G(Y)$ is a left adjoint, and the lax square 
    \[
      \begin{tikzcd}
        F(Y) \arrow{r} \arrow{d} & G(Y) \arrow{d}
        \arrow[dl,Rightarrow,inner sep=2pt]\\
        F(Y') \arrow{r} & G(Y')
      \end{tikzcd}
    \]
    commutes for all morphisms $Y \to Y'$ in $\mathcal{Y}$.  In this
    case the right adjoint is given by the mate of this square, which is also a lax square.
  \item A 1-morphism in $\FUN(\mathcal{Y}, \mathcal{X})_{\colax}$,
    \ie{} a colax natural
    transformation $\phi \colon F \to G$, is a right adjoint
    \IFF{} 
    for every $Y \in \mathcal{Y}$ the morphism $\phi_{Y} \colon F(Y)
    \to G(Y)$ is a right adjoint, and the colax square 
    \[
      \begin{tikzcd}
        F(Y) \arrow{r} \arrow{d} & G(Y) \arrow{d}
        \arrow[dl,Leftarrow,inner sep=2pt]\\
        F(Y') \arrow{r} & G(Y')
      \end{tikzcd}
    \]
    commutes for all morphisms $Y \to Y'$ in $\mathcal{Y}$.  In this
    case the left adjoint is given by the mate of this square, which
    is also a colax square.
  \end{enumerate}
\end{thm}

\begin{remark}
  Although we have not found a specific mention of the 2-categorical
  analogue of this statement, it can be seen as a special case of
  Kelly's theory of \emph{doctrinal adjunctions}
  \cite{KellyDoctrine}*{Theorem 1.4},
  at least if the target 2-category is cocomplete. A related result is
  \cite{StreetLaxConstr}*{Theorem 1}, which shows that a lax
  transformation of functors from a category to $\Cat$ has a left
  adjoint colax transformation \IFF{} each component has a left adjoint.
\end{remark}

\begin{proof}
  Statement (i) for the \itcat{} $\mathcal{X}$ is equivalent to
  statement (ii) for $\mathcal{X}^{\twop}$ using the equivalences of
  \cref{cor:FUNlax2op} and \cref{rmk:adjsymm}. It thus
  suffices to prove (i).

  We first suppose that $\phi$ is a left adjoint, so that there exists
  a lax natural transformation $\rho \colon G \to F$ that is its right
  adjoint, a unit $\eta \colon \id \to \rho \phi$ and a counit
  $\epsilon \colon \phi \rho \to \id$. Since any functor of \itcats{}
  preserves adjunctions, we then have that the component $\phi_{Y}
  \colon F(Y) \to G(Y)$ is a left adjoint in $\mathcal{X}$ with right
  adjoint $\rho_{Y}$, with unit and counit given by the components of
  $\eta$ and $\epsilon$ at $Y$. For a 1-morphism $f \colon Y \to Y'$,
  the lax transformations $\phi$ and $\rho$ supply lax squares
  \[
    \begin{tikzcd}
      F(Y) \arrow{r}{\phi_{Y}} \arrow{d}[swap]{F(f)} & G(Y) \arrow{d}{G(f)}
      \arrow[dl,Rightarrow,inner sep=2pt, "\phi(f)"]\\
      F(Y') \arrow{r}[below]{\phi_{Y'}} & G(Y'),
    \end{tikzcd}
    \qquad
    \begin{tikzcd}
      G(Y) \arrow{r}{\rho_{Y}} \arrow{d}[swap]{G(f)} & F(Y) \arrow{d}{F(f)}
      \arrow[dl,Rightarrow,inner sep=2pt, "\rho(f)"]\\
      G(Y') \arrow{r}[below]{\rho_{Y'}} & F(Y').
    \end{tikzcd}
  \]
  and the 2-morphisms $\eta$ and $\epsilon$ supply diagrams
  that amount to commutative diagrams of 1-morphisms
  \[
    \begin{tikzcd}
      {} & F(f)\rho_{Y}\phi_{Y} \arrow{d}{\rho(f)\phi_{Y}} \\
      F(f) \arrow{ur}{F(f)\eta_{Y}} \arrow{dr}[below left]{\eta_{Y'}F(f)}  & \rho_{Y'}G(f)\phi_{Y}
      \arrow{d}{\rho_{Y'}\phi(f)} \\
      & \rho_{Y'}\phi_{Y'}F(f),
    \end{tikzcd}
    \qquad
    \begin{tikzcd}
      G(f)\phi_{Y}\rho_{Y} \arrow{d}[swap]{\phi(f)\rho_{Y}}
      \arrow{dr}{G(f)\epsilon_{Y}} \\
      \phi_{Y'}F(f)\rho_{Y} \arrow{d}[swap]{\phi_{Y'}\rho(f)} & G(f) \\
      \phi_{Y'} \rho_{Y'} F(f) \arrow{ur}[below right]{\epsilon_{Y'}F(f)}
    \end{tikzcd}
    \]
  The second lax square has a \emph{mate square}, which is an oplax square
  \[
    \begin{tikzcd}
      F(Y) \arrow{r}{\phi_{Y}} \arrow{d}[swap]{F(f)} & G(Y) \arrow{d}{G(f)}
      \arrow[dl,Leftarrow,inner sep=2pt, "\psi(f)"]\\
      F(Y') \arrow{r}[below]{\phi_{Y'}} & G(Y'),
    \end{tikzcd}
  \]
  where $\psi(f) \colon \phi_{Y'}F(f) \to G(f)\phi_{Y}$ is the composite
  \[ \phi_{Y'}F(f) \xto{\phi_{Y'}F(f)\eta_{Y}}
    \phi_{Y'}F(f)\rho_{Y}\phi_{Y} \xto{\rho(f)}
    \phi_{Y'}\rho_{Y'}G(f)\phi_{Y} \xto{\epsilon_{Y'}G(f)\phi_{Y}} G(f)\phi_{Y}.\]
  We claim that $\psi(f)$ is an inverse to the 2-morphism
  $\phi(f)$. Indeed, using the equivalences of 2-morphisms from the unit
  and counit we get commutative diagrams
     \[
   \begin{tikzcd}[column sep=huge]
     \phi_{Y'}F(f) \arrow[bend left=20]{rrr}{\phi_{Y'}\eta_{Y'}F(f)} \arrow[bend
     right=10]{drr}[below left]{\psi(f)} \arrow{r}{\phi_{Y'}F(f)\eta_{Y}} &
     \phi_{Y'}F(f)\rho_{Y}\phi_{Y} 
     \arrow{r}{\phi_{Y'}\rho(f)\phi_{Y}} & \phi_{Y'}\rho_{Y'}G(f)\phi_{Y} \arrow{r}{\phi_{Y'}\rho_{Y'}\phi(f)}  \arrow{d}{\epsilon_{Y'}G(f)\phi_{Y}} &
     \phi_{Y'}\rho_{Y'}\phi_{Y'}F(f)
     \arrow{d}{\epsilon_{Y'}\phi_{Y'}F(f)} \\
      && G(f)\phi_{Y} \arrow{r}[below]{\phi(f)} & \phi_{Y'}F(f),
    \end{tikzcd}
  \]
  \[
    \begin{tikzcd}[column sep=huge]
  G(f)\phi_{Y} \arrow{r}{\phi(f)} \arrow{d}{G(f)\phi_{Y}\eta_{Y}} & \phi_{Y'}F(f)
      \arrow{d}{\phi_{Y'}F(f)\eta_{Y}} \arrow[bend left=10]{drr}{\psi(f)} \\
      G(f)\phi_{Y}\rho_{Y}\phi_{Y} \arrow[bend
      right=20]{rrr}{G(f)\phi_{Y}\epsilon_{Y}} \arrow{r}{\phi(f)\rho_{Y}\phi_{Y}}
      & \phi_{Y'}F(f)\rho_{Y}\phi_{Y} \arrow{r}{\phi_{Y'}\rho(f)\phi_{Y}} &
      \phi_{Y'}\rho_{Y'}G(f)\phi_{Y} \arrow{r}{\epsilon_{Y'}G(f)\phi_{Y}} & G(f)\phi_{Y}.
    \end{tikzcd}
  \]
Together with the adjunction equivalences these diagrams show that
$\psi(f)$ is inverse to $\phi(f)$, and so $\phi(f)$ is
invertible. Thus any left adjoint morphism in $\FUN(\mathcal{Y},
\mathcal{X})_{\lax}$ does indeed lie in $\FUN(\mathcal{Y},
\mathcal{X})$.

We now need to prove the converse, \ie{} if we have a natural
transformation $\phi \colon F \to G$ in
$\FUN(\mathcal{Y}, \mathcal{X})$ such that $\phi_{Y}$ is a left adjoint
for all $Y$, then $\phi$ is a left adjoint in $\FUN(\mathcal{Y},
\mathcal{X})_{\lax}$. Since the space of left adjoints in
$\FUN(\mathcal{Y}, \mathcal{X})_{\lax}$ commutes with colimits in
$\mathcal{Y}$, it suffices to show this for $\mathcal{Y}$ being
$C_{0}, C_{1}$, and $C_{2}$ (with the case of $C_{0}$ being trivial). 
For the case of $C_{1}$ we have a commutative square
\[
  \begin{tikzcd}
    A \arrow{r}{l} \arrow{d}{a}& B \arrow{d}{b} \\
    A' \arrow{r}{l'} & B'
  \end{tikzcd}
\]
given by an equivalence $\iota \colon bl \isoto l'a$ and
where $l$ and $l'$ are left adjoints, and we must show that this has a
right adjoint in $\FUN(C_{1},\mathcal{X})_{\lax}$. Let $r \colon b \to
a$ and $r' \colon b' \to a'$ be right adjoints of $l$ and $l'$,
and let
\[ \eta \colon \id_{A} \to rl, \quad \epsilon \colon lr \to \id_{B},\]
\[ \eta' \colon \id_{A'} \to r'l', \quad \epsilon' \colon l'r' \to \id_{B'},\]
be unit and counit 2-morphisms. The right adjoint
will be given by the mate square
\[
  \begin{tikzcd}
    B \arrow{r}{r} \arrow{d}{b}& A \arrow{d}{a}
    \arrow[dl,Rightarrow,inner sep=2pt, "\psi"] \\
    A' \arrow{r}{r'} & B',
  \end{tikzcd}
\]
where the
2-morphism $\psi \colon ar \to r'b$ is the composite
\[ ar \xto{\eta'ar} r'l'ar \xto{r'\iota^{-1}r} r'blr \xto{r'b\epsilon}
  r'b.\]
Composing the original square with the mate we get lax squares
\[
  \begin{tikzcd}[column sep=2cm]
    B \arrow{r}{lr} \arrow{d}{b}& B \arrow{d}{b}
    \arrow[dl,Rightarrow,inner sep=2pt, "(l'\psi)(\iota r)"] \\
    B' \arrow[swap]{r}{l'r'} & B',
  \end{tikzcd}
  \qquad
  \begin{tikzcd}[column sep=2cm]
    A \arrow{r}{rl} \arrow{d}{a}& A \arrow{d}{a}
    \arrow[dl,Rightarrow,inner sep=2pt, "(r'\iota)(\psi l)"] \\
    A' \arrow[swap]{r}{r'l'} & A'.
  \end{tikzcd}
\]
Using \cref{lem:graytenscolim}, to define the unit and counit we must
define diagrams of shape $[2](1,0) \cup_{C_{2}} [1]([1]^{2})
\cup_{C_{2}} [2](0,1)$ in $\mathcal{X}$. These are
given by using the units and counits of the two adjunctions together
with commutative squares of 2-morphisms of the form
\[  \begin{tikzcd}
    a \arrow{r}{a \eta} \arrow[equals]{d} & arl
    \arrow{d}{(r'\iota)(\psi l)} \\
    a \arrow{r}{\eta'a} & r'l'a
  \end{tikzcd}
  \qquad
  \begin{tikzcd}
    blr \arrow{r}{b\epsilon} \arrow{d}[left]{(l'\psi)(\iota r)} & b \arrow[equals]{d} \\
    l'r'b \arrow{r}{\epsilon' b} & b,
  \end{tikzcd}
\]
which can be defined as the commutative diagrams
\[
  \begin{tikzcd}
    a \arrow{rrr}{a\eta} \arrow{dr}{\eta'a}\arrow[equals]{dddd} & & & arl \arrow{d} \\
    & r'l'a \arrow{rr} \arrow{dr}{r'\iota^{-1}}  \arrow[equals,bend
    right]{dddrr} & & r'l'arl \arrow{d}{r'\iota^{-1}rl} \\
    & & r'bl \arrow{r} \arrow[equals]{dr} \arrow{ddr}[below
    left]{r'\iota}  & r'blrl \arrow{d} \\
    & & & r'bl \arrow{d}{r'\iota} \\
    a \arrow{rrr}{\eta'a} & & & r'l'a,
  \end{tikzcd}
  \qquad
  \begin{tikzcd}
    blr \arrow{rrr}{b\epsilon} \arrow{d}{\iota r} \arrow{ddr}{\iota r}
    \arrow[bend left,equals]{dddrr} & & & b
    \arrow[equals]{dddd} \\
    l'ar \arrow{d} \arrow[equals]{dr} \\
    l'r'l'ar \arrow{d}{l'r'\iota^{-1}r} \arrow{r} & l'ar
    \arrow[swap]{dr}{\iota^{-1}r} \\
    l'r'blr \arrow{rr} \arrow{d} & & blr \arrow{dr} \\
    l'r'b \arrow{rrr}{\epsilon'b} & & & b,
  \end{tikzcd}
  \]
composed of naturality squares together with the adjunction
equivalences for $l$ and $l'$
and the invertibility equivalence of $\iota$.

To check the adjunction identities it is convenient to first give an
alternative description of these diagrams: Recall that the unit and
counit of the adjunction $l \dashv r$ can be described as mates:
\[
  \begin{tikzcd}
    A \arrow[equals]{r} \arrow[equals]{d} & A \arrow{d}{l} \\
    A \arrow{r}{l} & B
  \end{tikzcd}
  \rightsquigarrow
  \begin{tikzcd}
    A \arrow[equals]{r} \arrow{d}[swap]{l} & A \arrow[equals]{d} \ar[dl,Rightarrow,inner sep=2pt, "\eta"] \\
    B \arrow{r}{r} & A
  \end{tikzcd}
  \qquad  
  \begin{tikzcd}
    A \arrow{r}{l} \arrow{d}[swap]{l} & B \arrow[equals]{d} \\
    B \arrow[equals]{r} & B
  \end{tikzcd}
  \rightsquigarrow
  \begin{tikzcd}
    B \arrow{r}{r} \arrow[equals]{d} & A \arrow{d}{l} \ar[dl,Rightarrow,inner sep=2pt, "\epsilon"] \\
    B \arrow[equals]{r} & B
  \end{tikzcd}
  \]
The diagrams for the unit and counit above can then be obtained by
taking mates horizontally in the following cubes:
\[
  \begin{tikzcd}[row sep=small]
    A \arrow{dd}{a} \arrow[equals]{rr} \arrow[equals]{dr} & & A
    \arrow{dr}{l} \arrow{dd}[near start]{a} \\
    & A \arrow[crossing over]{rr}[near start]{l} & & B \arrow{dd}{b} \\
      A' \arrow[equals]{rr} \arrow[equals]{dr} & & A' \arrow{dr}{l'}
      \\
       & A' \arrow[leftarrow,crossing over]{uu}[near end]{a} \arrow{rr}{l'} & & B'
     \end{tikzcd}
     \rightsquigarrow
  \begin{tikzcd}[row sep=small]
    A \arrow{dd}{a} \arrow[equals]{rr} \arrow{dr}{l} & & A
    \arrow[equals]{dr} \arrow{dd}[near start]{a} \ar[dl,Rightarrow,inner sep=2pt]
     \\
    & B \arrow[crossing over]{rr}[near start]{r} & & A \arrow{dd}{a}
    \ar[ddll,Rightarrow,inner sep=2pt,bend left=12,crossing over] \\
      A' \arrow[equals]{rr} \arrow{dr}{l'} & & A' \arrow[equals]{dr} \ar[dl,Rightarrow,inner sep=2pt]
      \\
       & B' \arrow[leftarrow,crossing over]{uu}[near end]{b} \arrow{rr}{r'} & & A',
     \end{tikzcd}     
  \]
\[
  \begin{tikzcd}[row sep=small]
    A \arrow{dd}{a} \arrow{rr}{l} \arrow{dr}{l} & & B 
    \arrow[equals]{dr} \arrow{dd}[near start]{b} \\
    & B \arrow[crossing over,equals]{rr} & & B \arrow{dd}{b} \\
      A' \arrow{rr}[near end]{l'} \arrow{dr}{l'} & & B' \arrow[equals]{dr}
      \\
       & B' \arrow[leftarrow,crossing over]{uu}[near end]{b} \arrow[equals]{rr} & & B'
     \end{tikzcd}
     \rightsquigarrow
  \begin{tikzcd}[row sep=small]
    B \arrow{dd}{b} \arrow{rr}{r} \arrow[equals]{dr} & & A
    \arrow{dr}{l} \arrow{dd}[near start]{a} \ar[dl,Rightarrow,inner
    sep=2pt] \ar[ddll,Rightarrow,inner sep=2pt,bend left=12]
     \\
    & B \arrow[crossing over,equals]{rr} & & B \arrow{dd}{b}
     \\
      B' \arrow{rr}[near end]{r'} \arrow[equals]{dr} & & A' \arrow{dr}{l'} \ar[dl,Rightarrow,inner sep=2pt]
      \\
       & B' \arrow[leftarrow,crossing over]{uu}[near end,swap]{b} \arrow[equals]{rr} & & B'.
     \end{tikzcd}     
   \]
With this description checking the adjunction identities amounts to showing that the following
composite cubes are horizontal and vertical identities, respectively (here we have omitted the
2-morphisms to make the diagram legible):
\[
  \begin{tikzcd}[row sep=tiny,column sep=tiny]
    A \arrow{ddd}{a} \arrow[equals]{rrrrr} \arrow{drr}{l} & & & & & A
    \arrow[equals]{drr} \arrow{ddd}{a} 
     \\
    & &  B \arrow[crossing over]{rrrrr}[near start]{r}
    \arrow[equals]{drr} & & & & & A \arrow{ddd}{a} \arrow{drr}{l}
    \\
    & & & & B \arrow[crossing over,equals]{rrrrr}   & &
    & & & B \arrow{ddd}{b} \\
    A' \arrow[equals]{rrrrr} \arrow{drr}{l'} & & & & & A' \arrow[equals]{drr}       \\
       & & B' \arrow[leftarrow,crossing over]{uuu}[near end]{b}
       \arrow{rrrrr}{r'} \arrow[equals]{drr} & & & & & A' \arrow{drr}{l'}\\
       & & & & B'\arrow[leftarrow,crossing over]{uuu}{b} \arrow[equals]{rrrrr} & & & & & B',
     \end{tikzcd}
     \qquad
  \begin{tikzcd}[row sep=small,column sep=small]
    B \arrow{dd}{b} \arrow{rr}{r} \arrow[equals]{dr} & & A \arrow[equals]{rr}
    \arrow{dr}{l} \arrow{dd}[near start]{a} & & A \arrow[equals]{dr}
    \arrow{dd}[near start]{a}
     \\
    & B \arrow[crossing over,equals]{rr} & & B 
    \arrow[crossing over]{rr}[near start]{r} & & A \arrow{dd}{a}
     \\
      B' \arrow{rr}[near end]{r'} \arrow[equals]{dr} & & A' \arrow[equals]{rr}
      \arrow{dr}{l'}  & & A' \arrow[equals]{dr} 
      \\
       & B' \arrow[leftarrow,crossing over]{uu}[near end,swap]{b}
       \arrow[equals]{rr} & & B' \arrow[leftarrow,crossing
       over]{uu}[near end,swap]{b} \arrow{rr}{r'} & & A'.
     \end{tikzcd}     
   \]
Since taking mates is compatible with horizontal and vertical
composition of squares, these composite cubes are obtained by taking
horizontal mates in the composite cubes
\[
  \begin{tikzcd}[row sep=tiny,column sep=tiny]
    A \arrow{ddd}{a} \arrow[equals]{rrrrr} \arrow[equals]{drr} & & & & & A
    \arrow{drr}{l} \arrow{ddd}{a} 
     \\
    & &  A \arrow[crossing over]{rrrrr}[near start]{l}
    \arrow{drr}{l} & & & & & B \arrow{ddd}{b} \arrow[equals]{drr}
    \\
    & & & & B \arrow[crossing over,equals]{rrrrr}   & &
    & & & B \arrow{ddd}{b} \\
    A' \arrow[equals]{rrrrr} \arrow[equals]{drr} & & & & & A' \arrow{drr}{l'}       \\
       & & A' \arrow[leftarrow,crossing over]{uuu}{a}
       \arrow{rrrrr}{l'} \arrow{drr}{l'} & & & & & B' \arrow[equals]{drr}\\
       & & & & B'\arrow[leftarrow,crossing over]{uuu}{b} \arrow[equals]{rrrrr} & & & & & B',
     \end{tikzcd}
     \qquad
  \begin{tikzcd}[row sep=small,column sep=small]
    A \arrow{dd}{a} \arrow[equals]{rr} \arrow[equals]{dr} & & A \arrow{rr}{l}
    \arrow{dr}{l} \arrow{dd}[near start]{a} & & B \arrow[equals]{dr}
    \arrow{dd}[near start]{b}
     \\
    & A \arrow[crossing over]{rr}[near start]{l} & & B 
    \arrow[equals,crossing over]{rr} & & B \arrow{dd}{b}
     \\
      A' \arrow[equals]{rr} \arrow[equals]{dr} & & A' \arrow{rr}[near start]{l'}
      \arrow{dr}{l'}  & & B' \arrow[equals]{dr} 
      \\
       & A' \arrow[leftarrow,crossing over]{uu}[near end,swap]{a}
       \arrow{rr}{l'} & & B' \arrow[leftarrow,crossing
       over]{uu}[near end,swap]{b} \arrow[equals]{rr} & & B'.
     \end{tikzcd}     
   \]
Since these composites are clearly identities, we have proved the
adjunction identities for $C_{1}$.

We now discuss the case $C_{2}$. Here the putative left adjoint is a
commutative cylindrical diagram
\[
\begin{tikzcd}
  A \ar[rr, "l"] \ar[d, bend right=45, "a"{left} ""{right,name=A1}] && B \ar[d,
  bend right=45, "b"{left}, ""{right,name=A2}] \ar[d, bend left=45,
  "b'"{right}, ""{left,name=B2}] \\
  A' \ar[rr, "l'"] 
  \ar[u, leftarrow, crossing over, bend right=45, "a'"{right}, ""{left,name=B1}]
  && B'
  \arrow[from=A1,to=B1,Rightarrow, "\alpha"]
  \arrow[from=A2,to=B2,Rightarrow, "\beta"]
\end{tikzcd}
\]
where $l$ and $l'$ are left adjoints. The commutativity data amounts
to equivalences
\[ bl \isoto l'a, \qquad b'l \isoto l'a'\]
together with a commutative square of morphisms $A \to B'$
\[
  \begin{tikzcd}
    bl \arrow{d}{\sim} \arrow{r}{\beta l} & b'l \arrow{d}{\sim}
    \\
    l'a \arrow{r}{l'\alpha} & l'a'.
  \end{tikzcd}
\]
The right adjoint is then given by a diagram of shape
\[
\begin{tikzcd}
  B \ar[rr,"r"] \ar[d, bend right=45, ""{right,name=A1}] && A \ar[d,
  bend right=45, ""{right,name=A2}] \ar[d, bend left=45, ""{left,name=B2}] \\
  B' \ar[rr,"r'"{below}]
  \ar[rru, bend right=30, 
  start anchor=north east, end anchor=south, crossing over, Leftarrow]
  \ar[rru, bend left=30, 
   start anchor=north, end anchor=south west,Leftarrow]
  \ar[u, leftarrow, crossing over, bend right=45, ""{left,name=B1}]
  && A'
  \arrow[from=A1,to=B1,Rightarrow]
  \arrow[from=A2,to=B2,Rightarrow]
\end{tikzcd}
\]
where $r$ and $r'$ are the right adjoints of $l$ and $l'$, and the
front and back lax squares are defined as above in the case
$C_{1}$. The additional coherence data required amounts to a
commutative square of morphisms, which we define as the composite of
the diagram
\[
  \begin{tikzcd}
    ar \arrow{r}{\alpha r} \arrow{d}{\eta'ar} & a'r \arrow{d}{\eta'
      a'r} \\
    r'l'ar \arrow{r}{r'l' \alpha r} \arrow{d}{\sim} & r'l'a' r
    \arrow{d}{\sim} \\
    r'blr \arrow{r}{r' \beta lr} \arrow{d}{r'b\epsilon} & r'b'lr
    \arrow{d}{r'b' \epsilon} \\
    r'b \arrow{r}{r'\beta} & r'b'
  \end{tikzcd}
\]
using the specified square and two naturality squares.

The decomposition in \cref{lem:graytenscolim} implies that the non-obvious part of defining
the unit and counit (using the unit and counit for $l \dashv r$ and
$l' \dashv r'$) is specifying two commutative cubes, which we can
define using naturality data for the (co)units defined in the
$C_{1}$-case; this can also be thought of as taking mates in one
direction in a 4-dimensional cube. Naturality of taking mates then gives the adjunction
identities by a 4-dimensional version of the argument above.
\end{proof}

\begin{cor}\label{cor:ADJlaxeq}
  Let $\FUN(C_{1}, \mathcal{X})_{\ladj}$ and $\FUN(C_{1},
  \mathcal{X})_{\radj}$ denote the full sub-\itcats{} of $\FUN(C_{1},
  \mathcal{X})$ containing only the left and right adjoint morphisms
  in $\mathcal{X}$, respectively. There are equivalences
  \[ \ADJ(\mathcal{X})_{\lax} \isoto \FUN(C_{1},
    \mathcal{X})_{\radj},\]
  \[ \ADJ(\mathcal{X})_{\colax} \isoto \FUN(C_{1},
    \mathcal{X})_{\ladj},\]
  given by composition with the morphisms $C_{1} \to \fadj$ 
  picking out the right and left adjoint 1-morphisms, respectively.
\end{cor}
\begin{proof}
  We prove the first equivalence, the proof of the second is
  similar. For any \itcat{} $\mathcal{Z}$ we have natural equivalences
  \[ \Map(\mathcal{Z}, \ADJ(\mathcal{X})_{\lax}) \simeq \Map(\fadj
    \otimes^{\lax} \mathcal{Z}, \mathcal{X}) \simeq \Map(\fadj,
    \FUN(\mathcal{Z}, \mathcal{X})_{\colax}).\]
  By \cref{thm:rvadj} evaluation at the right adjoint gives an
  equivalence
\[ \Map(\fadj,
    \FUN(\mathcal{Z}, \mathcal{X})_{\colax}) \isoto
  \Map(C_{1}, \FUN(\mathcal{Z},
  \mathcal{X})_{\colax})^{\radj},\] which by
  \cref{propn:laxfunadj} is equivalent to the space of
  natural transformations, \ie{} 1-morphisms in $\FUN(\mathcal{Z},
  \mathcal{X})$, that are levelwise right adjoints. In other words,
  this is precisely $\Map(\mathcal{Z}, \FUN(C_{1}, \mathcal{X})_{\radj})$.
\end{proof}

As an immediate consequence, we get the following naturality statement
for the process of taking mates of commutative squares:
\begin{cor}\label{cor:mateC1}
  There are natural functors of \itcats{}
  \[ \FUN(C_{1}, \mathcal{X})_{\radj} \isofrom
    \ADJ(\mathcal{X})_{\lax} \to \FUN(C_{1},
    \mathcal{X})_{\ladj,\lax},\]
  \[ \FUN(C_{1}, \mathcal{X})_{\ladj} \isofrom
    \ADJ(\mathcal{X})_{\colax} \to \FUN(C_{1},
    \mathcal{X})_{\radj,\colax}\] given on objects by passing to the
  other adjoint and on morphisms by taking mates. \qed
\end{cor}

\begin{remark}
  The procedure of taking mates for functors of \icats{} has
  previously been considered in \cite{LiuZheng} and
  \cite{BarwickMackey} in the case where the mate is invertible, and
  more generally in the book \cite{GaitsgoryRozenblyum}.
\end{remark}

\begin{remark}\label{rmk:adjcomp}
  Let $\fadj^{n} := \fadj \amalg_{[0]} \cdots \amalg_{[0]} \fadj$
  denote the \icat{} of $n$ composable adjunctions. Then
  \cref{thm:rvadj} furnishes an equivalence
  \[ \Map(\fadj^{n}, \mathcal{X}) \simeq \Map(\Delta^{n},
    \mathcal{X})^{\radj},\]
  where the right-hand side denotes the subspace of $\Map(\Delta^{n},
  \mathcal{X})$ of composable sequences of $n$ right adjoints. Since
  right adjoints are closed under composition, we have a simplicial
  space $\Map(\Delta^{\bullet}, \mathcal{X})^{\radj}$ natural in the
  \itcat{} $\mathcal{X}$, which by the Yoneda lemma implies that the
  representing objects $\fadj^{\bullet}$ also form a simplicial
  object. We can then upgrade \cref{cor:ADJlaxeq} to a
  natural equivalence
  \[ \FUN(\fadj^{\bullet}, \mathcal{X})_{\lax} \isoto
    \FUN(\Delta^{\bullet}, \mathcal{X})_{\radj}.\]
  This implies that taking mates is compatible with composition, since
  we get a composite functor
  \[ \FUN(\Delta^{\bullet}, \mathcal{X})_{\radj} \isofrom
    \FUN(\fadj^{\bullet}, \mathcal{X})_{\lax} \to
    \FUN((\Delta^{\bullet})^{\op}, \mathcal{X})_{\ladj,\lax},\]
  where we write $(\Delta^{\bullet})^{\op}$ to emphasize that the
  order of composition is reversed.
\end{remark}

\begin{remark}\label{rmk:Zagadj}
In his thesis~\cite{Zaganidis}, D.~Zaganidis considers for $I \in
\bTt$ the universal
2-category $\fadj(I)_{\lax}$ with an $I$-shaped diagram
of lax morphisms of adjunctions and shows that this also
satisfies the universal property
\[ \Map_{\CatIT}(\fadj(I)_{\lax}, \mathcal{X}) \simeq
  \Map_{\CatIT}(I, \FUN(C_{1}, \mathcal{X})_{\radj}).\]
It follows that $\fadj(I)_{\lax}$ is equivalent to the
Gray tensor product $\fadj \otimes^{\lax} I$. In fact, it is not
hard to see from the explicit definition of $\fadj(I)_{\lax}$ that
this is the classical Gray tensor product of $\fadj$ and $I$, and so
agrees with $\fadj \otimes^{\lax} I$ under
\cref{ass:gray}.  
\end{remark}

\section{Lax Morphisms of Monads}\label{sec:mnd}
We now turn to monads and (co)lax morphisms between them, which again
arise as (co)lax natural transformations.
\begin{notation}
  Let $\fmnd$ denote the full subcategory of $\mathfrak{adj}$ on the
  object $-$; this is the ``walking monad'' 2-category.
\end{notation}

\begin{defn}
  Let $\mathcal{X}$ be an \itcat{}. A \emph{monad} in
  $\mathcal{X}$ is a functor of \itcats{} $\fmnd \to \mathcal{X}$, and
  a \emph{(co)lax morphism} of monads is a (co)lax natural
  transformation between monads, \ie{} a functor
  \[ \fmnd \otimes^{\pcolax} C_{1} \to \mathcal{X}.\] We write
  $\MND(\mathcal{X})_{\pcolax} := \FUN(\fmnd, \mathcal{X})_{\pcolax}$
  for the \itcat{} of monads in $\mathcal{X}$ and (co)lax
  morphisms between them, and $\Mnd(\mathcal{X})_{\pcolax}$ for the
  underlying \icat{}.
\end{defn}

\begin{remark}
  For ordinary 2-categories, the notions of lax and colax morphisms of monads
  were first introduced by Street~\cite{StreetFormalMonad}, who called
  them \emph{monad functors} and \emph{monad opfunctors}.
\end{remark}

We will use results of Zaganidis to relate $\MND(\mathcal{X})_{\lax}$ to
\itcats{} of monadic adjunctions and monadic right adjoints; as
Zaganidis works in the framework of categories strictly enriched in
quasicategories developed by Riehl and Verity, this requires
the following additional assumption on the \itcat{} $\mathcal{X}$:
\begin{defn}
  We say an \itcat{} $\mathcal{X}$ is \emph{cosmifiable} if it can be
  modelled by a category strictly enriched in quasicategories that is
  an \emph{$\infty$-cosmos} in the sense of Riehl and Verity
  \cite{RiehlVerityBook}*{Definition 1.2.1}.
\end{defn}

\begin{exs}\label{ex:cosmex}\ 
  \begin{enumerate}[(i)]
  \item The \itcat{} $\CATI$ is cosmifiable (it can be modelled by the
    simplicial category of quasicategories).
  \item If $\mathcal{X}$ is cosmifiable, then so is
    $\mathcal{X}^{\twop}$ by \cite{RiehlVerityBook}*{Definition
      1.2.25}. In particular, $\CATI^{\twop}$ is cosmifiable.
  \end{enumerate}
\end{exs}

In his thesis~\cite{Zaganidis}, Zaganidis considers the full
sub-2-category $\fmnd(I)_{\lax}$ of $\fadj(I)_{\lax}$ (see~\cref{rmk:Zagadj})
consisting of an $I$-shaped diagram of lax morphisms of monads, which can be
identified with the classical Gray tensor product of $\fmnd$ and
$I$. Under
\cref{ass:gray}, this means that $\fmnd(I)_{\lax}$
corresponds to $\fmnd \otimes^{\lax} I$. We can then
state the main result of \cite{Zaganidis} as follows:
\begin{thm}[Zaganidis]\label{thm:Zaganidis}
  Let $\mathcal{X}$ be a cosmifiable \itcat{}. The restrictions
  \[ \Fun(\fadj \otimes^{\lax} I, \mathcal{X}) \rightarrow
    \Fun(\fmnd \otimes^{\lax} I, \mathcal{X})\]
  for $I \in \bTt$
  have fully faithful right adjoints, with image those functors $\fadj
  \otimes^{\lax} I \to \mathcal{X}$ where the underlying adjunction at each
  object of $I$ is monadic.
\end{thm}
\begin{remark}
  In the case $I = C_{0}$, the right adjoint
  \[\Fun(\fmnd, \mathcal{X}) \to \Fun(\fadj, \mathcal{X})\]
  is due to Riehl and Verity. Morally, this adjoint is given by an
  $(\infty,2)$-categorical right Kan extension along the inclusion
  $\fmnd \hookrightarrow \fadj$ and should thus exist for any target
  $(\infty,2)$-category where certain weighted limits exist. (Indeed,
  it seems plausible that the property of being ``cosmifiable''
  precisely amounts to the existence of a class of such limits.) The ideal proof
  of \cref{thm:Zaganidis} would then simply check that these weighted
  limits exist in $\FUN(I, \mathcal{X})_{\colax}$. However, the theory of
  weighted limits in $(\infty,2)$-categories has not yet been set
  up. Riehl and Verity circumvent this by modelling
  $(\infty,2)$-categories as categories enriched in quasicategories,
  and showing that certain ordinary weighted limits are homotopically
  meaningful, and Zaganidis applies the same technique to
  $\fmnd(I)_{\lax} \to \fadj(I)_{\lax}$.
\end{remark}

This has the following consequence:
\begin{cor}\label{cor:mndadj}
  Let $\mathcal{X}$ be a cosmifiable \itcat{}. The forgetful functor \[\ADJ(\mathcal{X})_{\lax} \to \MND(\mathcal{X})_{\lax}\]
  has a fully faithful right adjoint, with image the full sub-\itcat{}
  $\ADJ(\mathcal{X})_{\lax,\txt{mnd}}$ of monadic adjunctions. In
  particular, there are equivalences of \itcats{}
  \[ \MND(\mathcal{X})_{\lax} \isoto \ADJ(\mathcal{X})_{\lax,\txt{mnd}} \isoto
    \FUN(\Delta^{1}, \mathcal{X})_{\txt{mndradj}},\]
  where the latter denotes the full sub-\itcat{} of $\FUN(\Delta^{1},
  \mathcal{X})$ whose objects are the monadic right adjoints.
\end{cor}

\begin{remark}
  For monads on a single, fixed \icat{} $\mathcal{C}$ (interpreted as
  associative algebras in endomorphisms of $\mathcal{C}$) this
  comparison (in the case where $\mathcal{X}$ is $\CATI$) has
  previously been obtained by Heine~\cite{Heine} by a different method.
\end{remark}

\begin{remark}
  By \cref{ex:cosmex}(ii) if $\mathcal{X}$ is cosmifiable then so is
  $\mathcal{X}^{\twop}$. Here we have equivalences
  \[ \ADJ(\mathcal{X}^{\twop})_{\lax} \simeq \FUN(\fadj^{\twop},
    \mathcal{X})^{\twop}_{\colax} \simeq \ADJ(\mathcal{X})^{\twop}_{\colax},\]
  using the equivalence $\fadj^{\twop} \simeq \fadj$ of
  \cref{rmk:adjsymm} together with \cref{cor:FUNlax2op}, and
  \[ \MND(\mathcal{X}^{\twop})_{\lax} \simeq \FUN(\fmnd^{\twop},
    \mathcal{X})^{\twop}_{\colax} =:
    \txt{COMND}(\mathcal{X})^{\twop}_{\colax}. \]
  Note that here $\fmnd^{\twop}$ is equivalent to the full subcategory
  of $\fadj$ on the object $+$. Applying \cref{cor:mndadj} to
  $\mathcal{X}^{\twop}$ and reversing 2-morphisms we then see that
  the forgetful functor
  \[ \ADJ(\mathcal{X})_{\colax} \to
    \txt{COMND}(\mathcal{X})_{\colax} \]
  given by restriction to the object $+$ has a fully faithful
  \emph{left} adjoint with image the full sub-\itcat{} of
  \emph{comonadic} adjunctions.
\end{remark}

For the proof of \cref{cor:mndadj} we use the following observation:
\begin{lemma}\label{lem:laxadj}
  Suppose $\mathcal{X}$ is an \itcat{} and $\phi \colon \mathcal{A}
  \to \mathcal{B}$ is a functor of \itcats{} such that for every $I
  \in \bTt$ the induced functor
\[ (\phi \otimes^{\lax} I)^{*} \colon \Fun(\mathcal{B} \otimes^{\lax} I, \mathcal{X}) \to
  \Fun(\mathcal{A} \otimes^{\lax} I, \mathcal{X}) \]
has a right adjoint $R_{I}$, and for every morphism in $\bTt$ the mate square
for these adjoints commutes. Then the functor
\[ \phi^{*} \colon \FUN(\mathcal{B}, \mathcal{X})_{\lax} \to
  \FUN(\mathcal{A}, \mathcal{X})_{\lax} \]
has a right adjoint, given on objects by $R_{C_{0}}$, and with unit
and counit transformations given objectwise by the unit and counit for
$R_{C_{0}}$ and on morphisms by the unit and counit for $R_{C_{1}}$.
\end{lemma}
\begin{proof}
  The functors $(\phi \otimes^{\lax} I)^{*}$ are natural in $I$, and
  so give a morphism in $\Fun(\bTt^{\op}, \CatI)$, with $\phi^{*}$
  given as the induced morphism in $\Fun(\bTt^{\op}, \mathcal{S})$.
  The right adjoints assemble to a morphism in
  $\FUN(\bTt^{\op}, \CATI)_{\lax}$ by
  \cref{propn:laxfunadj}, but as the mate squares commute
  this a priori lax natural transformation is an
  ordinary natural transformation, and the units and counits determine
  an adjunction in $\FUN(\bTt^{\op}, \CATI)$. Passing to underlying
  $\infty$-groupoids, the right adjoints give a functor
  $R \colon \FUN(\mathcal{A}, \mathcal{X})_{\lax} \to
  \FUN(\mathcal{B}, \mathcal{X})_{\lax}$ given on objects by
  $R_{C_{0}}$, as required. To obtain the unit and counit
  transformations, observe that for \itcats{}
  $\mathcal{U},\mathcal{V}, \mathcal{W}$ there is a natural map
  $\mathcal{U} \otimes^{\lax} (\mathcal{V} \times \mathcal{W}) \to
  (\mathcal{U} \otimes^{\lax} \mathcal{V}) \times \mathcal{W}$, determined by
  the natural maps
  $\mathcal{U} \otimes^{\lax} (\mathcal{V} \times \mathcal{W}) \to
  \mathcal{U} \otimes^{\lax} \mathcal{V}$ and
  $\mathcal{U} \otimes^{\lax} (\mathcal{V} \times \mathcal{W}) \to
  \mathcal{U} \otimes^{\lax} \mathcal{W} \to \mathcal{W}$. The levelwise
  unit gives maps
  \[ \Map(I, \FUN(\mathcal{B}, \mathcal{X})_{\lax}) \simeq \Map(\mathcal{B} \otimes^{\lax} I, \mathcal{X}) \to \Map([1],
    \Fun(\mathcal{B} \otimes^{\lax} I, \mathcal{X}) \simeq \Map((\mathcal{B}
    \otimes^{\lax} I) \times [1], \mathcal{X}),\]
  natural in $I$, where we can now apply the map $\mathcal{B} \otimes^{\lax}
  (I \times [1]) \to (\mathcal{B} \otimes^{\lax} I) \times [1]$ to get
  a natural map
  \[ \Map(I, \FUN(\mathcal{B}, \mathcal{X})_{\lax}) \to \Map(\mathcal{B} \otimes^{\lax}
  (I \times [1]), \mathcal{X}) 
  \simeq \Map(I, \FUN([1],
  \FUN(\mathcal{B}, \mathcal{X})_{\lax})),\]
which corresponds to a functor of \itcats{}
\[ \FUN(\mathcal{B}, \mathcal{X})_{\lax} \times [1] \to
  \FUN(\mathcal{B}, \mathcal{X})_{\lax},\]
as required. By naturality the two diagrams
\[
  \begin{tikzcd}
    \mathcal{B} \otimes^{\lax} (I \times [0]) \arrow{r}{\sim} \arrow{d} &
    (\mathcal{B} \otimes^{\lax} I) \times [0] \arrow{d} \\
    \mathcal{B} \otimes^{\lax} (I \times [1]) \arrow{r} &
    (\mathcal{B} \otimes^{\lax} I) \times [1]     
  \end{tikzcd}
  \]
commute, which implies that this is a natural transformation from the
identity to $R\phi^{*}$. Similarly the levelwise counits give a natural
transformation
\[ \FUN(\mathcal{A}, \mathcal{X})_{\lax} \times [1] \to
  \FUN(\mathcal{A}, \mathcal{X})_{\lax}\]
from $\phi^{*}R$ to the identity. To show that this gives an
adjunction it suffices to check that the induced natural
transformations $R \to R$ and $\phi^{*} \to \phi^{*}$ are given by
equivalences for all objects and morphisms, which is clear since these
are then induced by the adjunction equivalences for $R_{C_{0}}$ and $R_{C_{1}}$.
\end{proof}

\begin{proof}[Proof of \cref{cor:mndadj}]
  Apply \cref{lem:laxadj} to the adjunctions from \cref{thm:Zaganidis}.
\end{proof}

\begin{cor}
  Let $\mathcal{X}$ be a cosmifiable \itcat{}. The inclusion \[\Fun([1], \iota_{1}\mathcal{X})_{\txt{mndradj}} \hookrightarrow
  \Fun([1], \iota_{1}\mathcal{X})_{\txt{radj}}\] has a left adjoint, which takes a
  right adjoint to the right adjoint of the associated monadic
  adjunction. \qed
\end{cor}

\begin{cor}\label{cor:Mndcocart}
  Let $\mathcal{X}$ be a cosmifiable \itcat{}. The functor $\Mnd(\mathcal{X})_{\lax} \to \iota_{1}\mathcal{X}$, taking a monad to the
  object it acts on, has cocartesian morphisms over morphisms in
  $\mathcal{X}$ that are right adjoints. If $T$ is a monad on $X$
  and $\rho \colon X \to Y$ is a morphism with left
  adjoint $\lambda$ then the cocartesian morphism over $\rho$ has
  target $\rho T \lambda$ and is given by the transformation $\rho T \lambda \rho \to
  \rho T$ coming from the counit.
\end{cor}
\begin{proof}
  We have a commutative diagram
  \[
    \begin{tikzcd}
      \Adj(\mathcal{X})_{\lax} \arrow{r}{\sim} \arrow{dr} & \Fun([1],
      \iota_{1}\mathcal{X})_{\txt{radj}} \arrow[hookrightarrow]{r}  \arrow{d}& \Fun([1], \iota_{1}\mathcal{X}).
      \arrow{dl}{\txt{ev}_{1}} \\
      & \iota_{1}\mathcal{X}.
    \end{tikzcd}
  \]
  Here ${\txt{ev}_{1}} \colon \Fun([1], \iota_{1}\mathcal{X}) \to \iota_{1}\mathcal{X}$ is a
  cocartesian fibration, with the cocartesian morphisms given by
  composition. Since a composite of right adjoints is a right adjoint,
  the full subcategory $\Fun([1],
  \iota_{1}\mathcal{X})_{\txt{radj}}$ has cocartesian morphisms over maps in
  $\iota_{1}\mathcal{X}$ that are right adjoints, hence the same is true for the
  equivalent \icat{} $\Adj(\mathcal{X})_{\lax}$.
  
  Now observe that we have a commutative triangle
  \[
    \begin{tikzcd}
    \Adj(\mathcal{X})_{\lax} \arrow{rr}{L} \arrow{dr} & & \Mnd(\mathcal{X})_{\lax}
    \arrow{dl} \\
     & \iota_{1}\mathcal{X},      
    \end{tikzcd}
  \]
  and that the right adjoint $R \colon \Mnd(\mathcal{X})_{\lax} \hookrightarrow
  \Adj(\mathcal{X})_{\lax}$ also commutes with the functors to $\iota_{1}\mathcal{X}$. In
  this situation $L$ necessarily takes a cocartesian morphism in
  $\Adj(\mathcal{X})_{\lax}$ to a cocartesian morphism in
  $\Mnd(\mathcal{X})_{\lax}$, hence $\Mnd(\mathcal{X})_{\lax}$ also has
  cocartesian morphisms over right adjoints in $\iota_{1}\mathcal{X}$. The
  description of the cocartesian morphisms in $\Mnd(\mathcal{X})_{\lax}$ now follows from the
  description of those in $\Adj(\mathcal{X})_{\lax}$.
\end{proof}

\section{Lax Morphisms of Endofunctors}
In this section we briefly consider endofunctors and (co)lax morphisms
between them, and the forgetful functor from our \itcats{} of monads
and (co)lax morphisms.
\begin{defn}
Let $\fend$ be the universal \icat{} with an endomorphism,
given by the pushout square
\[
  \begin{tikzcd}
    \partial C_{1} \arrow{r} \arrow{d} & C_{0} \arrow{d} \\
    C_{1} \arrow{r} & \fend.
  \end{tikzcd}
\]
Then $\fend$ can be identified with the 1-category
$B\mathbb{N}$ corresponding to the free monoid $\mathbb{N}$. If
$\mathcal{X}$ is an \itcat{}, we define
\[ \END(\mathcal{X})_{\pcolax}  := \FUN(\fend,
  \mathcal{X})_{\pcolax},\]
and write $\End(\mathcal{X})_{\pcolax}$ for the underlying \icat{}.
\end{defn}

\begin{remark}
  Since the Gray tensor product preserves colimits in each variable by
  \cref{ass:gray}, the pushout square above induces a
  pullback square of \itcats{}
\[
  \begin{tikzcd}
    \END(\mathcal{X})_{\pcolax} \arrow{r} \arrow{d} &
    \FUN(C_{1}, \mathcal{X})_{\pcolax} \arrow{d} \\
    \mathcal{X} \arrow{r}{\txt{diag}} & \mathcal{X} \times \mathcal{X},
  \end{tikzcd}
  \]
where the right vertical map is given by composition with the two
inclusions $C_{0} \hookrightarrow C_{1}$.
\end{remark}
  
\begin{remark}
There is also a functor $\mathfrak{end} \to \mathfrak{mnd}$ picking out the
underlying endofunctor of the universal monad, which induces a
commutative triangle
\[\begin{tikzcd}
  \MND(\mathcal{X})_{\pcolax} \arrow{dr} \arrow{rr} &
  & \END(\mathcal{X})_{\pcolax} \arrow{dl} \\
   & \mathcal{X}.
 \end{tikzcd}
\]
\end{remark}

In the case where $\mathcal{X}$ is the \itcat{} of \icats{}, we make
some observations on how the cocartesian morphisms of
\cref{cor:Mndcocart} behave in this triangle:

\begin{propn}\label{propn:EndMndloccoC}\ 
  \begin{enumerate}[(i)]
  \item The projection $\End(\CATI)_{\lax} \to \CatI$ has locally
    cocartesian morphisms and locally cartesian morphisms over
    functors that are right adjoints.
  \item The forgetful functor $\Mnd(\CATI)_{\lax} \to \End(\CATI)_{\lax}$
    preserves these locally cocartesian morphisms.
  \end{enumerate}
\end{propn}
\begin{proof}
  We first prove (i). Suppose $(\mathcal{C},P)$ and $(\mathcal{D}, Q)$
  are objects of
  $\End(\CATI)_{\lax}$ and $R \colon \mathcal{C} \to \mathcal{D}$ is a morphism
  with a left adjoint $L \colon \mathcal{D} \to \mathcal{C}$. Then
  \[\Map_{\End(\CATI)_{\lax}}((\mathcal{C},P), (\mathcal{D},Q))_{R} \simeq
  \Map_{\Fun(\mathcal{D}, \mathcal{C})}(QR, RP). \]
  From the adjunction identities it is immediate that this space is
  equivalent to $\Map_{\End(\mathcal{D})}(Q,RPL)$, so the
  natural transformation $RPLR \to RP$ coming from the counit $LR \to
  \id$ gives a locally cocartesian morphism over $R$ from $P$ to
  $RPL$. Similarly, the space is equivalent to
  $\Map_{\End(\mathcal{C})}(LQR, P)$, and the natural transformation
  $QR \to RLQR$ coming from the unit $\id \to RL$ gives a locally
  cartesian morphism.

  (ii) is now clear from the description of the locally cocartesian
  morphisms in \cref{cor:Mndcocart}.
\end{proof}

\begin{defn}
  Let $\CatI^{\txt{radj}}$ denote the subcategory of $\CatI$
  containing only the morphisms that are right adjoints. Then we
  define $\Mnd(\CATI)_{\lax}^{\radj}$ and $\End(\CATI)_{\lax}^{\radj}$
  by pulling back $\Mnd(\CATI)_{\lax}$ and $\End(\CATI)_{\lax}$ along
  the inclusion $\CatI^{\radj} \to \CatI$, \ie{} we restrict to those
  lax morphisms between monads and endofunctors whose underlying
  morphism in $\CatI$ is a right adjoint.
\end{defn}

\begin{cor}\label{cor:Mndradjcocart}
  There is a commuting triangle
  \opctriangle{\Mnd(\CATI)_{\lax}^{\radj}}{\End(\CATI)_{\lax}^{\radj}}{\CatI^{\txt{radj}},}{}{}{}
  where the two downward functors are cocartesian fibrations, and the
  horizontal functor preserves cocartesian morphisms. Moreover, the
  right-hand functor is also a cartesian fibration.
\end{cor}
\begin{proof}
  We know that the two downward functors are locally cocartesian
  fibrations, and that the horizontal functor preserves locally
  cocartesian morphisms.  It then suffices by
  \cite{HTT}*{Proposition~2.4.2.8} to show that the locally
  cocartesian morphisms in $\End(\CATI)_{\lax}^{\radj}$ are closed
  under composition, which is clear from our description of these
  morphisms. Similarly, the right-hand functor is a cartesian
  fibration.
\end{proof}

\section{Lax Transformations and Icons}\label{sec:icon}
Recall that, as we discussed in \S\ref{subsec:itcat}, we can view
$\CatIT$ as a full subcategory of $\Fun(\Dop, \CatI)$. The
straightening--unstraightening equivalence identifies the latter \icat{}
with the \icat{} $\Cat_{\infty/\Dop}^{\txt{cocart}}$ of cocartesian
fibrations over $\Dop$ and functors over $\Dop$ that preserve
cocartesian morphisms. This \icat{} has a natural enhancement to an \itcat{}
where the 2-morphisms are natural transformations over $\Dop$ between
such functors, and this restricts to a notion of 2-morphism between
functors of \itcats{} that is quite different from the usual notion of
2-morphisms as natural transformations.\footnote{For \itcats{}
  $\mathcal{X}, \mathcal{Y}$ a natural transformation is a morphism of
  \itcats{} $\mathcal{X} \times C_{1} \to
\mathcal{Y}$, or equivalently a morphism in
$\FUN(\mathcal{X},\mathcal{Y})$.} For example, if we view
monoidal \icats{} $\mathcal{M}, \mathcal{M}'$ as
(pointed) \itcats{} with one object, then a (pointed) functor of
\itcats{} $\mathcal{M} \to \mathcal{M}'$ is a monoidal functor, but a
natural transformation in the usual sense between two such functors
$F,G \colon \mathcal{M} \to \mathcal{M}'$ amounts to specifying an
object $x \in \mathcal{M}'$ and a natural equivalence
\[ F(\blank) \otimes x \simeq x \otimes G(\blank), \] while our new
notion of 2-morphism gives precisely the monoidal natural
transformations.

Our goal in this section is to identify these ``new'' 2-morphisms with
certain colax natural transformations, namely those given at each
object by an equivalence. This amounts to an \icatl{} version of a
result of Lack~\cite{Lack:icons}, who refers to this class of colax
transformations as ``icons''.\footnote{An acronym for ``Identity
  Component Oplax Natural transformations''.}  In \S\ref{sec:mndalg} we will use this result to
identify two \icats{} of monads on a fixed \icat{}.

It is convenient to view our new 2-morphisms in terms of a certain
tensoring of $\CatIT$ over $\CatI$, which we will now define:
\begin{defn}
  The \icat{}
  $\Fun(\Dop, \CatI)$ is tensored over $\CatI$ by taking products with
  constant functors, \ie{} for $X \in \Fun(\Dop, \CatI)$ and
  $\mathcal{C} \in \CatI$ we can define $X \times \mathcal{C}$ as the
  functor $[n] \mapsto X_{n}\times \mathcal{C}$. This preserves
  colimits in both variables, so we have a cotensoring
  $X^{\mathcal{C}}$ (given by $[n] \mapsto \Fun(\mathcal{C}, X_{n})$)
  and an enrichment $\Nat(X, Y)$ in $\CatI$, satisfying
  \[ \Map_{\Fun(\Dop, \CatI)}(X \times \mathcal{C}, Y) \simeq
    \Map_{\Fun(\Dop, \CatI)}(X, Y^{\mathcal{C}}) \simeq
    \Map_{\CatI}(\mathcal{C}, \Nat(X, Y)).\]
\end{defn}

\begin{remark}\label{rmk:Natcoc}
  If we view $\Fun(\Dop, \CatI)$ as cocartesian fibrations to $\Dop$,
  then morphisms in $\Nat(X, Y)$ indeed correspond to natural
  transformations between functors over $\Dop$ that preserve
  cocartesian morphisms.
\end{remark}

If $\mathcal{X}$ and $\mathcal{Y}$ are
$(\infty,2)$-categories, we obtain an \icat{} $\Nat(\mathcal{X},
\mathcal{Y})$ whose objects are functors $\mathcal{X} \to
\mathcal{Y}$. The precise result we will prove in this section is the
 following description of this \icat{}:
\begin{thm}\label{thm:icon}
  Given $(\infty,2)$-categories $\mathcal{X}$ and $\mathcal{Y}$ there
  are functors
  \[ \Nat(\mathcal{X}, \mathcal{Y}) \to \Fun(\mathcal{X},
    \mathcal{Y})_{\colax},\]
  \[ \Nat(\mathcal{X}, \mathcal{Y})^{\op} \to \Fun(\mathcal{X},
    \mathcal{Y})_{\lax},\] where the first identifies
  $\Nat(\mathcal{X}, \mathcal{Y})$ with the subcategory of
  $\Fun(\mathcal{X}, \mathcal{Y})_{\colax}$ whose morphisms are the
  colax natural transformations that are given by equivalences in
  $\mathcal{Y}$, and similarly in the lax case. More precisely, we have natural pullback squares
  \[
    \begin{tikzcd}
      \Nat(\mathcal{X}, \mathcal{Y}) \arrow{r} \arrow{d} &
      \Fun(\mathcal{X}, \mathcal{Y})_{\colax} \arrow{d} \\
      \Map(\iota_{0}\mathcal{X}, \iota_{0}\mathcal{Y}) \arrow{r} &
      \Fun(\iota_{0}\mathcal{X}, \mathcal{Y}),
    \end{tikzcd} \quad
        \begin{tikzcd}
      \Nat(\mathcal{X}, \mathcal{Y})^{\op} \arrow{r} \arrow{d} &
      \Fun(\mathcal{X}, \mathcal{Y})_{\lax} \arrow{d} \\
      \Map(\iota_{0}\mathcal{X}, \iota_{0}\mathcal{Y}) \arrow{r} &
      \Fun(\iota_{0}\mathcal{X}, \mathcal{Y}),
    \end{tikzcd}
  \]
  where in the left-hand square the right vertical map is given by composition with the
  inclusion $\iota_{0}\mathcal{X} \to \mathcal{X}$ (via the
  equivalence $\Fun(\iota_{0}\mathcal{X}, \mathcal{Y})_{\colax} \simeq
  \Fun(\iota_{0}\mathcal{X},\mathcal{Y})$ of \cref{rmk:colaxigpd}) and
  the bottom horizontal map is given by composition with the inclusion $\iota_{0}\mathcal{Y}
 \to \mathcal{Y}$ (and the equivalence $\Fun(\iota_{0}\mathcal{X},
 \iota_{0}\mathcal{Y}) \simeq \Map(\iota_{0}\mathcal{X}, \iota_{0}\mathcal{Y})$),
  and similarly in the right-hand square.
\end{thm}

To prove \cref{thm:icon}, we start with the following observation:
\begin{propn}
  If $X$ lies in one of the subcategories $\Cat_{(\infty,2)}$,
  $\Seg^{\mathcal{S}}_{\Dop}(\CatI)$ or $\Seg_{\Dop}(\CatI)$, then so
  does $X^{\mathcal{C}}$ for any $\mathcal{C} \in \CatI$. 
\end{propn}
\begin{proof}
  If $X \in \Seg_{\Dop}(\CatI)$, then the Segal map
  \[ X_{n}^{\mathcal{C}} \to X_{1}^{\mathcal{C}}
    \times_{X_{0}^{\mathcal{C}}} \cdots \times_{X_{0}^{\mathcal{C}}}
    X_{1}^{\mathcal{C}}\] is just $\Fun(\mathcal{C}, \blank)$ applied
  to the Segal map for $X$, and so is an equivalence since
  $\Fun(\mathcal{C}, \blank)$ preserves limits. Moreover, for
  $X \in \Seg_{\Dop}^{\mathcal{S}}(\CatI)$ the
  \icat{}
  \[X^{\mathcal{C}}_{0} \simeq \Fun(\mathcal{C}, X_{0}) \simeq
    \Map_{\mathcal{S}}(\|\mathcal{C}\|, X_{0}) \] is an
  $\infty$-groupoid.

  Now suppose $X \in \Cat_{(\infty,2)}$; we must show that the Segal
  space $(X^{\mathcal{C}})^{\simeq} \simeq \Map_{\CatI}(\mathcal{C}, X)$
  is complete. Since complete Segal spaces are closed under limits, it
  suffices to consider the cases where  $\mathcal{C}$ is $\Delta^{0}$
  (where we just get $X^{\Delta^{0}} \simeq X$)
  and $\Delta^{1}$. Here $(X^{\Delta^{1}})^{\simeq}_{0} \simeq X_{0}$ and
  $(X^{\Delta^{1}})^{\simeq}_{1} \simeq \Map(\Delta^{1}, X_{1})$. A
  morphism in $X^{\Delta^{1}}$ is thus a 2-morphism in $X$, \ie{} a
  diagram of shape
  \[
    \begin{tikzcd}
      \bullet 
      \arrow[d, bend left=40, ""{name=B,inner sep=1pt,left}]
      \arrow[d, bend right=40, ""{name=A,inner sep=1pt,right}]      
      \\
      \bullet
      \arrow[from=A,to=B,Rightarrow]
    \end{tikzcd}
  \]
  in $X$, with composition in $X^{\Delta^{1}}$ given by composing two
  such diagrams vertically. For such a morphism to be an equivalence
  it follows that both the two 1-morphisms and the 2-morphism must be
  an equivalence in $X$, which proves completeness.
\end{proof}

\begin{remark}\label{rmk:str2cattens}
  If $\mathbf{X}$ is a strict 2-category, we can explicitly
  describe $\mathbf{X}^{[n]}$, which is again a strict 2-category:
  \begin{itemize}
  \item the objects are the objects of $\mathbf{X}$
  \item a morphism from $x$ to $y$ consists of $n+1$ morphisms $f_{i}
    \colon x \to y$ ($i=0,\ldots,n$), and 2-morphisms $f_{0} \to f_{1} \to \cdots \to
    f_{n}$, \ie{} a functor $[n] \to \mathbf{X}(x,y)$,
  \item a 2-morphism between two morphisms from $x$ to $y$, given by
    $f_{0} \to \cdots \to f_{n}$ and $g_{0} \to \cdots \to g_{n}$,
    consists of a commutative diagram of 2-morphisms of the form
    \[
      \begin{tikzcd}
        f_{0} \arrow{r} \arrow{d} & f_{1} \arrow{r} \arrow{d} & \cdots
        \arrow{r} & f_{n} \arrow{d} \\
        g_{0} \arrow{r} & g_{1} \arrow{r} & \cdots \arrow{r} & g_{n},
      \end{tikzcd}
    \]
    \ie{} a functor $[n] \times [1] \to \mathbf{X}(x,y)$,
  \item composition of morphisms and 2-morphisms is given in terms of
    composition in $\mathbf{X}$ in the evident way.
  \end{itemize}
\end{remark}

\begin{cor}
  There exists a functor
  \[ \blank\odot\blank \colon \Cat_{(\infty,2)} \times \CatI \to
    \Cat_{(\infty,2)},\]
  which preserves colimits in each variable, and satisfies
  \[ \Map_{\Cat_{(\infty,2)}}(\mathcal{X} \odot \mathcal{C},
    \mathcal{Y}) \simeq \Map_{\Cat_{(\infty,2)}}(\mathcal{X},
    \mathcal{Y}^{\mathcal{C}}) \simeq \Map_{\CatI}(\mathcal{C},
    \Nat(\mathcal{X}, \mathcal{Y}))\]
  for all \itcats{} $\mathcal{X}$ and $\mathcal{Y}$ and \icats{} $\mathcal{C}$.\qed
\end{cor}

\begin{remark}\label{rmk:odotop}
  If we view an \itcat{} $\mathcal{X}$ as an object of $\Fun(\Dop,
  \CatI)$, then $\mathcal{X}^{\twop}$ is obtained by taking opposite
  \icats{} levelwise. For any \icat{} $\mathcal{C}$ we therefore have 
  a natural equivalence
  \[ (\mathcal{X}^{\mathcal{C}})^{\twop} \simeq
    (\mathcal{X}^{\twop})^{\mathcal{C}^{\op}},\]
  given levelwise by the equivalence
  \[\Fun(\mathcal{C},\mathcal{X}_{i})^{\op} \simeq
    \Fun(\mathcal{C}^{\op},\mathcal{X}_{i}^{\op}).\]
  This translates into a natural equivalence
  \[ (\mathcal{X} \odot \mathcal{C})^{\twop} \simeq
    \mathcal{X}^{\twop} \odot \mathcal{C}^{\op}.\]
\end{remark}

The morphisms in Theorem~\ref{thm:icon} will arise from
natural transformations
\[ \mathcal{X} \otimes^{\colax} \mathcal{C} \to \mathcal{X} \odot
  \mathcal{C},\]
\[ \mathcal{X} \otimes^{\lax} \mathcal{C} \to \mathcal{X} \odot
  \mathcal{C}^{\op}.\]
In order to define these we require an explicit description of $I \odot [n]$ for
$I \in \bTt$ and $[n] \in \simp$. To obtain this we first define an explicit functor
$\Phi \colon \bTt \times \simp \to \CatIT$ (in fact taking values in
strict 2-categories) and check that this satisfies the co-Segal
condition in each variable. Then we will use the universal property of
$\odot$ to define a natural transformation $\odot \to \Phi$ and prove
that this is an equivalence.

\begin{defn}\label{defn:Phi}
  We define $\Phi$ on objects by
  \[ \Phi( [k](n_1,\ldots,n_k) , [m]) := [k]([n_{1}] \times
    [m],\ldots,[n_{k}] \times [m]),\]
  in the notation of Definition~\ref{defn:[n]free}. This has as objects $0,\ldots, k$ and as Hom-categories ($0 \leq i \leq 
j \leq k$) the posets
\[
\Hom(i,j) :=  \prod_{i < s \leq j} \left([n_s] \times [m]\right) \cong I(i,j)
\times [m]^{j-i},
\]
with composition given by isomorphism. In particular, we can view this
as a strict 2-category. Recall that if $I = [k](n_1,\ldots,n_k)$ and
$J= [l](n'_{1},\ldots,n'_{l})$ then a morphism $F \colon I \to J$ in $\bTt$ is
given by a morphism $\phi \colon [k] \to [l]$ in $\simp$ together with
morphisms $\psi_{ij} \colon [n_{i}] \to [n'_{j}]$ in $\simp$ whenever
$\phi(i-1)<j \leq \phi(i)$. If we are also given a morphism
$\mu \colon [m] \to [m']$, then the corresponding functor
\[ \Phi(F,\mu) \colon \Phi(I,[m]) \to 
  \Phi(J, [m']) \]
is given on objects by $i \mapsto \phi(i)$ and on morphism categories
by the functor
\[ \prod_{i < s \leq j} [n_{s}] \times [m] \to \prod_{\phi(i)<t \leq
    \phi(j)}  [n'_{t}] \times [m'] \]
that in the component indexed by $t$ is given by 
\[ \prod_{i < s \leq j} [n_{s}]  \times [m] \to  [n_{r}] \times [m]
  \xto{\psi_{rt} \times \mu} [n'_{t}] \times [m'],\]
where $r$ is the unique index such that $\phi(r-1)< t \leq
\phi(r)$. In other words, this is the functor
\[ I(i,j) \times [m]^{j-i} \to J(\phi(i),\phi(j)) \times
  [m']^{\phi(j)-\phi(i)} \]
given by $F(i,j)$ in the first variable and in the second by
\[
  \begin{split}
  \mu_{i,j} \colon \prod_{i < s \leq j}
  [m] \to \prod_{i < s \leq j} [m]^{\phi(s)-\phi(s-1)} \xto{\prod_{i
      < s \leq j} \mu^{\phi(s)-\phi(s-1)}} & \prod_{i < s \leq j}
  [m']^{\phi(s)-\phi(s-1)} \\ & \cong 
   \prod_{i < s \leq j} \prod_{\phi(s-1)<t\leq \phi(s)} [m'] \\ & \cong
   \prod_{\phi(i)<t \leq \phi(j)} [m'],     
  \end{split}
\]
 where the first functor is given by the diagonals of $[m]$.
\end{defn}

\begin{ex}
  For $\Phi(C_1, [1]) = 
  \Phi([1](0), [1])$ we get $[1](1) = C_{2}$, and for
  $\Phi(C_{2},[1])$ we get $[1]([1] \times [1])$, which we can think
  of as a ``suspension''
  of the commutative square: it has two objects $0,1$ and a commutative
  square of morphisms from $0$ to $1$.
\end{ex}

\begin{lemma}
  $\Phi$ satisfies the co-Segal condition in each variable.
\end{lemma}
\begin{proof}
  Lemma~\ref{lem:[n]Segal} immediately gives the first co-Segal condition
  in the $\bTt$-variable:
  \[
    \begin{split}
      \Phi([k](n_{1},\ldots,n_{k}), [m]) & \simeq [k]([n_{1}] \times [m],
    \ldots, [n_{k}] \times [m]) \\ & \simeq [1]([n_{1}]\times [m])
    \amalg_{[0]} \cdots \amalg_{[0]} [1]([n_{k}]\times [m]) \\ & \simeq
    \Phi([1](n_{1}), [m]) \amalg_{\Phi([0], [m])} \cdots
    \amalg_{\Phi([0], [m])} \Phi([1](n_{k}), [m]).
    \end{split}
  \]
  Moreover, Lemma~\ref{lem:[1]wccolim} gives the other co-Segal
  condition in the $\bTt$-variable:
  \[
    \begin{split}
    \Phi([1](n), [m]) & \simeq [1]([n] \times [m]) \simeq [1]([1] \times
    [m]) \amalg_{[1]([0] \times [m])} \cdots \amalg_{[1]([0] \times
      [m])} [1]([1] \times [m]) \\
    & \simeq [1]([1] \times [m]) \amalg_{[1](m)} \cdots
    \amalg_{[1](m)} [1]([1] \times [m])\\
    & \simeq \Phi(C_{2}, [m]) \amalg_{\Phi(C_{1},[m])} \cdots
    \amalg_{\Phi(C_{1}, [m])} \Phi(C_{2}, [m]).
    \end{split}
  \]
  Similarly, Lemma~\ref{lem:[1]wccolim} gives the co-Segal condition in
  the $\simp$-variable for $\Phi([1](n), \blank)$ from which the
  general case follows using the first decomposition.
\end{proof}

\begin{construction}
We will now define a natural transformation
\[
\eta_{I,m} \colon I \odot [m] \to \Phi(I, [m])
\]
(for $I \in \bbTheta_2$ and $[m] \in \simp$).
Since $I \odot [m]$ was defined by adjunction, to give this map
is equivalent to giving
\[
\eta'_{I,m} \colon I \to \Phi(I,[m])^{[m]}
\]
where the latter is cotensoring with $[m]$. Applying the description
of \cref{rmk:str2cattens} to
$\Phi(I, [m])^{[m]}$ when $I=[k](n_{1},\ldots,n_{k})$, we see that
\begin{itemize}
\item its objects are $0,\ldots,k$,
\item a morphism from $i$ to $j$ is a functor
  \[ [m] \to [m]^{j-i} \times I(i,j), \]
  (where $I(i,j) = \prod_{i < s \leq j} [n_s]$ if $i < j$ and $\emptyset$ if $i > j$)
\item a 2-morphism of functors $i$ to $j$ is a functor
  \[ [m] \times [1] \to [m]^{j-i} \times I(i,j). \]
\end{itemize}
In other word, $\Phi(I,[m])^{[m]}(i,j)$ is the functor category
$\Fun([m], \Phi(I,[m])(i,j))$.
The functor $\eta'_{I,m}$ is defined to be the identity on objects,
and for $i \leq j$ the functor of morphism categories
\[ I(i,j) \to \Phi(I,[m])^{[m]}(i,j) = \Fun([m], [m]^{j-i} \times
  I(i,j)) \]
is defined to be the one adjoint to the functor
\[ [m] \times I(i,j) \xto{\eta''_{I,m}(i,j)} [m]^{j-i} \times I(i,j) \]
given by the product of the diagonal $[m] \to [m]^{j-i}$ and the
identity of $I(i,j)$. Since the composition in $\Phi(I,[m])^{[m]}$ is
defined in terms of composition in $I$, it is evident that this
defines a functor of strict 2-categories. To show that $\eta$ is
natural, we must check that for $F \colon I \to J$ and $\alpha \colon
[m] \to [k]$ the square
\[
  \begin{tikzcd}
    I \odot [m] \arrow{r}{\eta_{I,m}} \arrow{d}{F \odot \alpha} &
    \Phi(I,[m]) \arrow{d}{\Phi(F,\alpha)} \\
    J \odot [k] \arrow{r}{\eta_{J,k}} & \Phi(J,[k])
  \end{tikzcd}
\]
commutes, which is equivalent to the commutativity of the diagram
\[
  \begin{tikzcd}
    I \arrow{dd}[swap]{F} \arrow{r}{\eta'_{I,m}}& \Phi(I,[m])^{[m]}
    \arrow{dr}{\Phi(F,\alpha)^{[m]}} \\
    & & \Phi(J,[k])^{[m]}. \\
    J \arrow{r}{\eta'_{J,k}} & \Phi(J,[k])^{[k]} \arrow{ur}[swap]{\Phi(J,[k])^{\alpha}}
  \end{tikzcd}
\]
This in turn amounts to the commutativity of the squares
\[
  \begin{tikzcd}[column sep=large]
    I(i,j) \times [m] \arrow{d}{F(i,j) \times \alpha}
    \arrow{r}{\eta''_{I,m}(i,j)} & I(i,j) \times [m]^{j-i}
    \arrow{d}{F(i,j) \times \alpha_{i,j}} \\
    J(Fi,Fj) \times [k] \arrow{r}{\eta''_{J,k}(Fi,Fj)} & J(Fi,Fj)
    \times [k]^{Fj-Fi},
  \end{tikzcd}
  \]
where $\alpha_{i,j}$ is defined as at the end of \cref{defn:Phi},
which is clear from the definitions of these functors.
\end{construction}

\begin{propn}\label{propn:Phieq}
  The natural $2$-functor
  \[ \eta_{I,m} \colon I \odot \Delta^m \longrightarrow
    \Phi(I,\Delta^m) \]
  is an equivalence.
\end{propn}

\begin{proof}[Proof of Proposition~\ref{propn:Phieq}]
  Since both $\odot$ and $\Phi$ are co-Segal, it is enough to establish the
  equivalence on generators $C_{i} \in \bTt$, $i =0,1,2$ and $[j] \in
  \simp$, $j = 0,1$. There are two non-trivial cases: it suffices to
  prove that the maps
  \[ \eta_{C_{1},1} \colon C_{1} \odot [1] \to \Phi(C_{1},[1]) \cong
    [1](1) = C_{2},\]
  \[ \eta_{C_{2},1} \colon C_{2} \odot [1] \to \Phi(C_{2},[1]) \cong
    [1]([1]\times [1]),\]
  are equivalences.

  For the first case, note that by adjunction we
  have (for any $2$-category $\mathcal{X}$)
  \[
  \Map(C_1 \odot [1], \mathcal{X}) \simeq \Map(C_1, 
  \mathcal{X}^{[1]}).\]  These are the $1$-morphisms in 
  $\mathcal{X}^{[1]}$, and if we think of $\mathcal{X}$ as an
  object of $\Fun(\Dop, \CatI)$ these are the objects of $\mathcal{X}_1^{[1]}$,
  which are the $2$-morphisms in $\mathcal{X}$. Thus $C_1 \odot
  [1]  \simeq C_2$, and the functor $C_{1} \to C_{2}^{[1]}$
  adjoint to the equivalence is the one picking out the non-trivial
  2-morphism in $C_{2}$, which is indeed $\eta'_{C_{1},1}$.  
  
  For the second case, we have
  \[\Map(C_2 \odot [1], \mathcal{X}) \simeq
  \Map(C_2, \mathcal{X}^{[1]}) \simeq \Map([1], 
  \mathcal{X}_1^{[1]}) \simeq \Map([1] \times [1],
  \mathcal{X}_1).\]  Since $[1] \times [1]$ is weakly
  contractible, by Lemma~\ref{lem:[1]wcrep} this
  is the same thing as $\Map([1]([1] \times [1]), \mathcal{X})$. Thus
  $C_{2} \odot [1] \simeq [1]([1] \times [1])$, and unwinding the
  definitions we see that the map \[C_{2} \to [1]([1] \times
  [1])^{[1]}\] adjoint to the equivalence is the one picking out the
  non-trivial 2-morphism in $[1]([1] \times
  [1])^{[1]}$ (corresponding to the commutative square of non-trivial 2-morphisms
  in $[1]([1] \times [1])$), which is indeed $\eta'_{C_{2},1}$.
\end{proof}

\begin{construction}
  We now define a natural functor of strict 2-categories
  \[ \nu_{I,m} \colon I \otimes^{\colax} [m] \to \Phi(I,[m]), \]
  for $I \in \bTt$, using the explicit description of $I
  \otimes^{\colax} [m]$ as a strict 2-category from \cref{defn:GrayTh2}.
  Let $I= [k](n_1,\ldots,n_k)$. On objects, the map $\nu_{I,m}$ is simply the 
  projection $\obj([k])\times \obj([m]) \to \obj([k])$. On hom categories, for fixed objects
  $(x,i)$ and $(x',i')$, we need to give a map of posets
  \[
    \operatorname{MaxCh}([x,x'] \times [i,i']) \times \!
    I(x,x') \to 
    [m]^{x'-x} \times I(x,x').\]
  We define this map to send a maximal chain in $[x,x'] \times [i,i']$
  to the tuple of column indices where the vertical steps are taken.
  By unwinding the definitions involved we see that this indeed
  defines a natural functor of 2-categories.  
\end{construction}

\begin{propn}\label{propn:odotpo}
  There is a natural pushout square of \itcats{}
  \[
    \begin{tikzcd}
      \iota_{0} \mathcal{X} \times \mathcal{C} \arrow{r} \arrow{d} &
      \mathcal{X} \otimes^{\colax} \mathcal{C} \arrow{d} \\
      \iota_{0} \mathcal{X} \times \|\mathcal{C}\| \arrow{r} & \mathcal{X} \odot \mathcal{C}.
    \end{tikzcd}
  \]
\end{propn}

\begin{proof}
  From the definition of $\nu_{I,m}$ we have for $I \in \bTt$ and $[m]
  \in \simp$ a natural commutative square
  \[
    \begin{tikzcd}
      \iota_{0}I \times [m] \arrow{r} \arrow{d} & I \otimes^{\colax}
      [m] \arrow{d}{\nu_{I,m}} \\
      \iota_{0}I \arrow{r} & \Phi(I,[m]).
    \end{tikzcd}
  \]
  We can extend this by colimits to a commutative square in
  $\Seg_{\bTt^{\op}}(\mathcal{S})$ of the correct form for
  $\mathcal{X} \in \Seg_{\bTt^{\op}}(\mathcal{S})$ and
  $\mathcal{C} \in \Seg_{\Dop}(\mathcal{S})$, which induces the
  required square in $\CatIT$ for $\mathcal{X} \in \CatIT$,
  $\mathcal{C} \in \CatI$ after completion.

  To see that this is a pushout square in $\CatIT$ it suffices to
  check the original square is a pushout in
  $\Seg_{\bTt^{\op}}(\mathcal{S})$ for all $I \in \bTt$ and
  $[m] \in \simp$, and since the functors satisfy the co-Segal
  condition in each variable to this it is enough to check the cases
  where $I = C_{0},C_{1},C_{2}$ and $m = 0,1$.  The cases involving
  $C_0$ and $[0]$ are trivial, so we are left with two non-trivial cases:
  $\Phi(C_1,[1]) \cong C_{2}$ and $\Phi(C_{2},[1]) \cong [1]([1]
  \times [1])$. Using the colimit decomposition of
  \cref{lem:graytenscolim} the square in the first case is
  \[
    \begin{tikzcd}
      C_{1} \amalg C_{1} \arrow{d} \arrow{r} & {[2]}(0,0) \cup_{C_{1}}
      C_{2} \cup_{C_{1}} [2](0,0) \arrow{d} \\
      C_{0} \amalg C_{0} \arrow{r} & C_{2},
    \end{tikzcd}
  \]
  which is a pushout since we have
  \[ [2](0,0) \cup_{C_{1}} C_{0} \simeq (C_{1} \cup_{C_{0}} C_{1})
    \cup_{C_{1}} C_{0} \simeq C_{1}.\]
  In the second case the square is
  \[
    \begin{tikzcd}
      C_{1} \amalg C_{1} \arrow{d} \arrow{r} & {[2]}(1,0) \cup_{C_{2}}
    [1]([1]^{2}) \cup_{C_{2}} {[2]}(0,1) \arrow{d} \\
      C_{0} \amalg C_{0} \arrow{r} & {[1]}([1]^{2}),
    \end{tikzcd}
  \]
  which is a pushout since we have
  \[ [2](1,0) \cup_{C_{1}} C_{0} \simeq (C_{2} \cup_{C_{0}} C_{1})
    \cup_{C_{1}} C_{0} \simeq C_{2}\]
  and similarly for $[2](0,1)$.
\end{proof}

\begin{remark}\label{rmk:odotlaxpo}
  Combining \cref{propn:odotpo} with the equivalences of
  \cref{rmk:odotop} and \cref{propn:gray2op}, we also obtain (since
  the functor $(\blank)^{\twop}$ preserves colimits) natural pushout
  squares
  \[
    \begin{tikzcd}
      \iota_{0} \mathcal{X} \times \mathcal{C} \arrow{r} \arrow{d} &
      \mathcal{X} \otimes^{\lax} \mathcal{C} \arrow{d} \\
      \iota_{0} \mathcal{X} \times \|\mathcal{C}\| \arrow{r} & \mathcal{X} \odot \mathcal{C}^{\op}.
    \end{tikzcd}
  \]  
\end{remark}

The proof of Theorem~\ref{thm:icon} is now immediate:
\begin{proof}[Proof of \cref{thm:icon}]
  In the colax case the pushout square of \cref{propn:odotpo} gives
  natural pullback squares
  \[
    \begin{tikzcd}
    \Map(\mathcal{X} \odot \mathcal{C}, \mathcal{Y}) \arrow{r}
    \arrow{d} & \Map(\iota_{0}\mathcal{X} \times \|\mathcal{C}\|,
    \mathcal{Y}) \arrow{d} \\
    \Map(\mathcal{X} \otimes^{\colax} \mathcal{C}, \mathcal{Y})
    \arrow{r} & \Map(\iota_{0} \mathcal{X} \times \mathcal{C}, \mathcal{Y}),
    \end{tikzcd}
  \]
  which we can rewrite using various adjunctions as
    \[
    \begin{tikzcd}
    \Map(\mathcal{C}, \Nat(\mathcal{X},\mathcal{Y})) \arrow{r}
    \arrow{d} & \Map(\mathcal{C}, \Map(\mathcal{X}, \mathcal{Y})) \arrow{d} \\
    \Map(\mathcal{C}, \Fun(\mathcal{X},\mathcal{Y})_{\colax})
    \arrow{r} & \Map(\mathcal{C}, \Fun(\iota_{0} \mathcal{X}, \mathcal{Y})),
    \end{tikzcd}
  \]
  which gives the desired pullback square in $\CatI$ by the Yoneda
  Lemma. In the lax case we proceed in the same way, using instead the
  variant pushout square of \cref{rmk:odotlaxpo}.
\end{proof}

\section{Monads as Algebras}\label{sec:mndalg}
In \S\ref{sec:adj} we considered monads as functors of \itcats{} from
$\fmnd$. Alternatively, we can view monads in an \itcat{}
$\mathcal{X}$ as associative algebras in the monoidal \icats{} of
endomorphisms of the objects of $\mathcal{X}$; this is the point of
view taken by Lurie in \cite{HA}*{\S 4.7}. Our goal in this section is
to compare these two approaches, and in particular use the results of
\S\ref{sec:icon} to relate the natural notions of morphisms in the two
cases.

\begin{remark}\label{rmk:iconmndend}
  Applying Theorem~\ref{thm:icon} to monads and endofunctors, we get (as
$\iota_{0}\fmnd \simeq \iota_{0}\fend \simeq C_{0}$) pullback squares of \icats{}
\[
  \begin{tikzcd}
    \Nat(\fmnd, \mathcal{X})^{\op} \arrow{r} \arrow{d} &
    \Mnd(\mathcal{X})_{\lax} \arrow{d} \\
    \iota_{0}\mathcal{X} \arrow{r} & \iota_{1}\mathcal{X},
  \end{tikzcd}
  \qquad
  \begin{tikzcd}
    \Nat(\fmnd, \mathcal{X}) \arrow{r} \arrow{d} &
    \Mnd(\mathcal{X})_{\colax} \arrow{d} \\
    \iota_{0}\mathcal{X} \arrow{r} & \iota_{1}\mathcal{X},
  \end{tikzcd}
\]
\[
  \begin{tikzcd}
    \Nat(\fend, \mathcal{X})^{\op} \arrow{r} \arrow{d} &
    \End(\mathcal{X})_{\lax} \arrow{d} \\
    \iota_{0}\mathcal{X} \arrow{r} & \iota_{1}\mathcal{X},
  \end{tikzcd}
  \qquad
  \begin{tikzcd}
    \Nat(\fend, \mathcal{X}) \arrow{r} \arrow{d} &
    \End(\mathcal{X})_{\colax} \arrow{d} \\
    \iota_{0}\mathcal{X} \arrow{r} & \iota_{1}\mathcal{X}.
  \end{tikzcd}
\]
In particular, for an object $X \in \mathcal{X}$ we can identify the fibre
$\Mnd(\mathcal{X})_{\colax,X}$ with the fibre $\Nat(\fmnd,
\mathcal{X})_{X}$, and similarly in the other three cases. 
\end{remark}

To describe these fibres, we therefore want to give an explicit
description of the \icats{}
$\Nat(\fmnd, \mathcal{X})$ and $\Nat(\fend, \mathcal{X})$. By
Remark~\ref{rmk:Natcoc}, we can view $\Nat(\mathcal{X}, \mathcal{Y})$
in terms of the corresponding cocartesian fibrations
$\intDop\mathcal{X}$ and $\intDop\mathcal{Y}$ as the \icat{}
$\Fun_{/\Dop}^{\txt{cocart}}(\intDop \mathcal{X}, \intDop
\mathcal{Y})$ of functors over $\Dop$ that preserve cocartesian
morphisms and natural transformations between these. To describe this
in the case of interest we recall some notions related to generalized
non-symmetric \iopds{} (see for instance \cite{enr} for more details):
\begin{defn}
  Recall that a double \icat{} is a functor $\Dop \to \CatI$
  satisfying the Segal condition. (Thus an \itcat{} can be described
  as a double \icat{} whose value at $[0]$ is an $\infty$-groupoid.)
  Equivalently, a double \icat{} is a cocartesian fibration over $\Dop$
  corresponding to such a functor; we write $\txt{Dbl}_{\infty}$ for
  the \icat{} of double \icats{}. We also write
  $\txt{Opd}^{\txt{ns,gen}}_{\infty}$ for the \icat{} of generalized
  non-symmetric \iopds{}, which are also certain \icats{} over
  $\Dop$. In particular, a generalized non-symmetric \iopd{} has
  cocartesian morphisms over inert maps in $\Dop$, and the
  morphisms in $\txt{Opd}^{\txt{ns,gen}}_{\infty}$ are the morphisms
  over $\Dop$ that preserve these inert cocartesian morphisms. If
  $\mathcal{O}$ is a generalized non-symmetric \iopd{}, then so is
  $\mathcal{O} \times \mathcal{I}$ for any \icat{} $\mathcal{I}$, and
  we write $\Alg_{\mathcal{O}}(\mathcal{P})$ for the \icat{} given by
  the complete Segal space
  $\Map_{\txt{Opd}^{\txt{ns,gen}}_{\infty}}(\mathcal{O} \times
  \Delta^{\bullet}, \mathcal{P})$. (We abbreviate the \icat{}
  $\Alg_{\Dop}(\mathcal{P})$ of associative algebras to just $\Alg(\mathcal{P})$.)
\end{defn}

\begin{defn}
  Any
  double \icat{} is a generalized non-symmetric \iopd{}, so that there
  is a forgetful functor
  \[ \txt{Dbl}_{\infty} \to \txt{Opd}^{\txt{ns,gen}}_{\infty}.\]
  This has a left adjoint $\txt{Env}$, the \emph{double envelope},
  given by a simple explicit formula (see \cite{nmorita}*{\S
    A.8}):
  \[ \txt{Env}(\mathcal{O}) \simeq \mathcal{O} \times_{\Dop}
    \txt{Act}(\Dop) \] where $\txt{Act}(\Dop)$ is the full subcategory
  of $\Fun(C_{1}, \Dop)$ spanned by the active maps, the fibre product
  uses the map to $\Dop$ given by evaluation at $0 \in C_{1}$, and the
  map $\txt{Env}(\mathcal{O}) \to \Dop$ is given by evaluation at
  $1 \in C_{1}$.  From this formula we see that
  $\txt{Env}(\mathcal{O} \times \mathcal{I}) \simeq
  \txt{Env}(\mathcal{O}) \times \mathcal{I}$, so that if $\mathcal{M}$
  is a double \icat{} and $\mathcal{O}$ is a generalized non-symmetric
  \iopd{} then the adjunction induces an equivalence
  \[ \Alg_{\mathcal{O}}(\mathcal{M}) \simeq
    \Fun_{/\Dop}^{\txt{cocart}}(\txt{Env}(\mathcal{O}), \mathcal{M}).\]
\end{defn}
\begin{remark}\label{rmk:env}
  Note in particular that $\txt{Env}(\mathcal{O})_{0} \simeq
  \mathcal{O}_{0}$ (as the only active map to $[0]$ in $\Dop$ is
  $\id_{[0]}$) while $\txt{Env}(\mathcal{O})_{1} \simeq
  \mathcal{O}^{\txt{act}}$, the subcategory of $\mathcal{O}$
  containing only the active maps (as every object in $\Dop$ has a
  unique active map to $[1]$). Thus if $\mathcal{O}_{0} \simeq C_{0}$
  (\ie{} $\mathcal{O}$ is a non-symmetric \iopd{}) then
  $\txt{Env}(\mathcal{O})$ is a monoidal \icat{} given by a monoidal
  structure on $\mathcal{O}^{\txt{act}}$. If we think of objects of
  $\mathcal{O}$ as lists $(X_{1},\ldots, X_{n})$ of objects $X_{i} \in
  \mathcal{O}_{1}$, then this monoidal structure is given by
  concatenation,
  \[ (X_{1},\ldots,X_{n}) \otimes (Y_{1},\ldots,Y_{m}) \simeq
    (X_{1},\ldots,X_{n},Y_{1},\ldots,Y_{m}).\]
\end{remark}

\begin{propn}\label{propn:mndenv}
  There are equivalences $\intDop \fmnd \simeq \txt{Env}(\Dop)$ and $\intDop\fend
  \simeq \txt{Env}(\Dop_{\txt{int}})$, and hence for $\mathcal{X}$ an
  \itcat{} there are natural equivalences
\[ \Nat(\fmnd, \mathcal{X}) \simeq
  \Alg_{\Dop}(\intDop\mathcal{X}),\]
\[ \Nat(\fend, \mathcal{X}) \simeq
  \Alg_{\Dop_{\txt{int}}}(\intDop\mathcal{X}) \simeq
  \mathcal{X}_{1} \times_{\mathcal{X}_{0} \times \mathcal{X}_{0}}
  \mathcal{X}_{0},\]\
under which the functor $\Nat(\fmnd, \mathcal{X}) \to \Nat(\fend,
\mathcal{X})$ corresponds to that given by composition with
$\Dop_{\xint} \to \Dop$.
\end{propn}
\begin{proof}
  By Remark~\ref{rmk:env} we know that $\txt{Env}(\Dop)$ is a
  monoidal structure on $(\Dop)^{\txt{act}} \simeq \simp_{+}$ given by
  concatenation, \ie{} it is precisely $(\simp_{+}, \star)$ which is
  the monoidal category corresponding to the one-object 2-category
  $\fmnd$. Similarly, $\txt{Env}(\Dop_{\txt{int}})$ is a monoidal
  structure on $(\Dop_{\txt{int}})^{\txt{act}}$, which is (as only the
  identity maps are both active and inert) the set $\{0,1,\ldots\}$,
  given by addition, which is precisely the monoidal category
  corresponding to the one-object 2-category $\fend$.
\end{proof}

\begin{defn}
  Let $i \colon C_{0} \to \Dop$ denote the functor picking out the
  object $[0]$. Then right Kan extension along $i$ gives a functor
  $i_{*}\colon \CatI \to \Fun(\Dop, \CatI)$ with $(i_{*}\mathcal{C})_{n}
  \simeq \mathcal{C}^{\times n+1}$. This is a double \icat{}, so we
  get an adjunction
  \[ i^{*} : \txt{Dbl}_{\infty} \rightleftarrows \CatI : i_{*}.\]
  We write $\Dop_{\mathcal{C}} \to \Dop$ for the cocartesian fibration
  corresponding to $i_{*}\mathcal{C}$. Note that, since $i^{*}$
  preserves products, we get for any double \icat{} $\mathcal{M}$ an
  equivalence
  \[ \Nat(\mathcal{M}, i_{*}\mathcal{C}) \simeq \Fun(\mathcal{M}_{0}, \mathcal{C}).\]
  Hence for any generalized
  non-symmetric \iopd{} $\mathcal{O}$ we get equivalences
  \[ \Alg_{\mathcal{O}}(\Dop_{\mathcal{C}}) \simeq
    \Fun^{\txt{cocart}}_{/\Dop}(\txt{Env}(\mathcal{O}),
    \Dop_{\mathcal{C}}) \simeq \Fun(\txt{Env}(\mathcal{O})_{0},
    \mathcal{C}) \simeq \Fun(\mathcal{O}_{0}, \mathcal{C}).\]
\end{defn}

\begin{defn}
  Suppose $\mathcal{M}$ is a double \icat{}, viewed as a cocartesian
  fibration. The unit of the adjunction $i^{*} \dashv i_{*}$
  corresponds to a functor $\mathcal{M} \to \Dop_{\mathcal{M}_{0}}$
  that preserves cocartesian morphisms. For $X \in \mathcal{M}_{0}$ we
  define $\mathcal{M}^{\otimes}_{X}$ to be the pullback
  \[
    \begin{tikzcd}
      \mathcal{M}^{\otimes}_{X} \arrow{r} \arrow{d} & \mathcal{M}
      \arrow{d} \\
      \Dop \arrow{r} & \Dop_{\mathcal{M}_{0}},
    \end{tikzcd}
  \]
  where the bottom horizontal map corresponds to $(\Dop)_{0} \simeq
  \{X\} \to \mathcal{M}_{0}$. This is a pullback of cocartesian
  fibrations over $\Dop$, so the natural projection
  $\mathcal{M}^{\otimes}_{X} \to \Dop$ is a cocartesian fibration,
  which exhibits $\mathcal{M}^{\otimes}_{X}$ as a monoidal structure
  on $(\mathcal{M}^{\otimes}_{X})_{1}$, which is the fibre of
  $\mathcal{M}_{1} \to \mathcal{M}_{0} \times \mathcal{M}_{0}$ at
  $(X,X)$.   For any generalized non-symmetric \iopd{} $\mathcal{O}$,
  the functor $\Alg_{\mathcal{O}}(\blank)$ preserves limits, and so
  gives a pullback square of \icats{}
  \[
    \begin{tikzcd}
      \Alg_{\mathcal{O}}(\mathcal{M}^{\otimes}_{X}) \arrow{r}
      \arrow{d} & \Alg_{\mathcal{O}}(\mathcal{M}) \arrow{d}\\
      \{X\} \arrow{r} & \Fun(\mathcal{O}_{0}, \mathcal{M}_{0}).
    \end{tikzcd}
  \]
\end{defn}
\begin{remark}
  If $\mathcal{X}$ is an \itcat{}, then the monoidal \icat{}
  $\mathcal{X}^{\otimes}_{X}$ is the monoidal structure on the \icat{}
  $\mathcal{X}(X,X)$ of endomorphisms of $X$ given by composition.
\end{remark}

Applying this construction to monads and endofunctors via
Remark~\ref{rmk:iconmndend} and Proposition~\ref{propn:mndenv}, we get:
\begin{cor}
  Let $\mathcal{X}$ be an \itcat{} and consider the commutative
  triangle
    \[
    \begin{tikzcd}
      \Mnd(\mathcal{X})^{\pcolax} \arrow{rr} \arrow{dr} &
      & \End(\mathcal{X})^{\pcolax} \arrow{dl} \\
      & \iota_{1}\mathcal{X}.
    \end{tikzcd}
  \]
  For any object $X \in \mathcal{X}$
  we have equivalences of fibres
  \[
    \Mnd(\mathcal{X})_{\colax,X} \simeq
    \Alg(\mathcal{X}^{\otimes}_{X}),
    \qquad
    \End(\mathcal{X})_{\colax,X} \simeq
        \mathcal{X}(X,X),\]
  \[\Mnd(\mathcal{X})_{\lax,X} \simeq
    \Alg(\mathcal{X}^{\otimes}_{X})^{\op},
    \qquad
    \End(\mathcal{X})_{\lax,X} 
     \simeq \mathcal{X}(X,X)^{\op}. \]
  Moreover, the morphisms on fibres at $X$ in the triangle
  can be identified with the forgetful functors
  \[ \Alg(\mathcal{X}^{\otimes}_{X})^{(\op)} \to
    \mathcal{X}(X,X)^{(\op)}\]
  from associative algebras to their underlying objects.\qed
\end{cor}

In \S\ref{sec:adj} we used results of Zaganidis and Riehl--Verity to
construct a fully faithful functor
\[  \Mnd(\CATI)_{\lax} \to \Fun(\Delta^{1}, \CatI) \]
with image the monadic right adjoints, and we just saw that the fibre
of $\Mnd(\CATI)_{\lax}$ at a fixed \icat{}
$\mathcal{C}$ is equivalent to the \icat{}
$\Alg(\End(\mathcal{C}))$ of associative algebras
in endofunctors of $\mathcal{C}$ under composition. Our functor thus
restricts to a fully faithful functor $\Alg(\End(\mathcal{C}))
\to \Cat_{\infty/\mathcal{C}}$ with image the monadic right adjoints. 
We denote the image of this functor at $T$ by $\Alg_{T}(\mathcal{C})$.

On the other hand, the monoidal \icat{} $\End(\mathcal{C})$ acts on
$\mathcal{C}$, so given a monoid $T \in \Alg(\End(\mathcal{C}))$ we
can consider the \icat{} $\Alg_{T}^{\txt{Lur}}(\mathcal{C})$ of left $T$-modules
in $\mathcal{C}$; the forgetful functor $\Alg_{T}^{\txt{Lur}}(\mathcal{C})\to
\mathcal{C}$ is proved by Lurie~\cite{HA} to be a monadic right
adjoint. The following proposition shows that these two monadic right
adjoints associated to $T$ are equivalent:
\begin{propn}
   For $T \in \Alg(\End(\mathcal{C}))$ there is a canonical
equivalence \[ 
  \Alg_{T}(\mathcal{C}) \simeq \Alg_{T}^{\txt{Lur}}(\mathcal{C}).\]
\end{propn}
\begin{proof}
  In the construction of Riehl and Verity the \icat{}
  $\Alg_{T}(\mathcal{C})$ and the monadic right adjoint
  $U_{T} \colon \Alg_{T}(\mathcal{C}) \to \mathcal{C}$ are obtained from the
  enriched right Kan extension of the functor
  $T \colon \fmnd \to \CATI$ along the inclusion
  $\fmnd \hookrightarrow \fadj$. From the structure of $\fadj$ we see
  that $U_{T}$ is a left $T$-module in
  $\Fun(\Alg_{T}(\mathcal{C}), \mathcal{C})$. Now by \cite{HA}*{Proposition 4.7.3.3} any right adjoint
functor $G \colon \mathcal{D} \to \mathcal{C}$ has an \emph{endomorphism
monad}, meaning a terminal object in the \icat{}
$\LMod(\Fun(\mathcal{D},\mathcal{C}))_{G}$ of monads $S$ on
$\mathcal{C}$ together with an $S$-action on $G$. Moreover, such a monad
$S$ acting on $G$ is the endomorphism monad of $G$ \IFF{} the composite
\[ S \to SGF \to GF \]
is an equivalence, where the first morphism uses the unit of the
adjunction and the second the action of $S$ on $G$.

The action of $T$ on $U_{T}$ certainly has this property, so $T$ is
the endomorphism monad of $U_{T}$. Now as $U_{T}$ is a monadic right
adjoint, Lurie's version of the Barr--Beck Theorem for \icats{}
\cite{HA}*{Theorem 4.7.3.5} (together with \cite{HA}*{Definition
  4.7.3.4}) furnishes an equivalence
\[ \Alg_{T}(\mathcal{C}) \isoto \Alg^{\txt{Lur}}_{T}(\mathcal{C}) := \LMod_{T}(\mathcal{C})\]
over $\mathcal{C}$.
\end{proof}

 \end{document}